\documentclass[11pt,english]{amsart}
\usepackage{amsfonts,amsmath,latexsym,verbatim,amscd,mathrsfs,color,array}
\usepackage{color}
\usepackage{marginnote}
\usepackage{graphicx}

\newtheorem{theorem}{Theorem}[section]
\newtheorem{proposition}{Proposition}[section]
\newtheorem{lemma}{Lemma}[section]
\newtheorem{definition}{Definition}[section]
\newtheorem{remark}{Remark}[section]
\newtheorem{notation}{Notation}[section]

\newcommand{\RR}{\mathbb R}

\newcommand{\del}{\partial}

\newcommand{\e}{\varepsilon}

\newcommand{\Ric}{\mathrm{Ric}}

\newcommand{\calB}{{\mathcal B}}
\newcommand{\calC}{{\mathcal C}}

\newcommand{\calE}{{\mathcal E}}

\newcommand{\calH}{{\mathcal H}}

\newcommand{\calK}{{\mathcal K}}

\newcommand{\calO}{{\mathcal O}}

\newcommand{\calR}{{\mathcal R}}
\newcommand{\calS}{{\mathcal S}}

\newcommand{\frakH}{{\mathfrak H}}

\newcommand{\tr}{{\mathrm{tr}}\,}

\newcommand{\codim}{\operatorname{codim}}

\title{Higher codimension isoperimetric problems}

\author{Rafe Mazzeo}
\thanks{R.M. Supported by NSF-DMS-1105050}
\address{Stanford University}
\email{mazzeo@math.stanford.edu}

\author{Frank Pacard}
\address{Centre de Math\'ematiques Laurent Schwartz, \'Ecole Polytechnique-CNRS}
\email{frank.pacard@math.polytechnique.fr}

\author{Tatiana Zolotareva}
\address{Centre de Math\'ematiques Laurent Schwartz, \'Ecole Polytechnique-CNRS}
\email{zolotareva@math.polytechnique.fr}

\begin{document}

\maketitle

\begin{abstract}
We consider a variational problem for submanifolds $Q \subset M$ with nonempty boundary $\del Q = K$.  We propose the definition that the boundary $K$ of any critical point $Q$ have constant mean curvature, which seems to be a new perspective when $\dim Q < \dim M$. We then construct small nearly-spherical solutions of this higher codimension CMC problem; these concentrate near the critical points of a certain curvature function. 
\end{abstract} 

\section{Introduction} 

\let\thefootnote\relax\footnotetext{F. Pacard and T. Zolotareva are partially supported by the ANR-2011-IS01-002 grant. }
\let\thefootnote\relax\footnotetext{T. Zolotareva is partially supported by FMJH foundation.}

Constant mean curvature (CMC) hypersurfaces are critical points of the area functional subject to a volume constraint.  Examples include sufficiently smooth solutions to the isoperimetric problem. If $K$ is an embedded submanifold in a Riemannian manifold $(M^{m+1},g)$, then its mean curvature vector $H_K$ is the trace of its shape operator. When $K$ is a hypersurface, then we say that $K$ has CMC if this vector  has constant length, and this is the only sensible definition in this case. However, when $\codim K > 1$, it is less obvious how to formulate the CMC condition, since there is more than one way one might regard the mean curvature vector as being constant. One definition that has perhaps received the most attention is to require that $H_K$ be parallel. This is quite restrictive, and for that reason, not very satisfactory.

We propose here a different, and directly variational, definition. Building on ideas of Almgren \cite{A}, and extending one standard characterization of CMC hypersurfaces, we define constant mean curvature submanifolds to be boundaries of submanifolds which are critical for a certain energy functional. Roughly speaking, we say that $K$ has constant mean curvature if $K = \del Q$ where $Q$ is minimal, $K$ has CMC in $Q$, and $H_K$ has no 
component orthogonal to $Q$. 

The goal of this paper is to show that generic metrics on any compact manifold admit `small' CMC submanifolds in this sense. The result proved here is a generalization of a well-known theorem by Ye \cite{Ye}, which constructs families of CMC hypersurfaces which are small perturbations of geodesic spheres centered at nondegenerate critical points of the scalar curvature function $\calR$. The more recent paper \cite{Pac-Xu} obtains such families of CMC hypersurfaces under general condition on the scalar curvature and in particular when it is constant; 
in that case, these hypersurfaces are centered near critical points of a different curvature invariant.  These various results illustrate 
the sense in which the metric must be generic: some scalar function of the curvature must have nondegenerate critical points.

Let us now introduce the relevant curvature function. For any $(k+1)$-dimensional subspace $\Pi_p \subset T_pM$, define 
the partial scalar curvature
\[
\mathcal R_{k+1} (\Pi_p): =  - \sum_{i,j=1}^{k+1} \langle R(E_i, E_j)E_i, E_j \rangle. 
\]
where $E_1, \ldots, E_{k+1}$ is any orthonormal basis for $\Pi_p$. Note that $\calR_{m+1}(T_pM)$ is the standard scalar curvature at $p$, while $\calR_2(\Pi_p)$ is twice the sectional curvature of the $2$-plane $\Pi_p$.  The Grassman bundle $G_{k+1}(TM)$ is the fibre bundle over $M$ with fibre at $p \in M$ the Grassmanian of all $(k+1)$-planes in $T_pM$. We regard $\calR_{k+1}$ as a smooth function on $G_{k+1}(M)$.

Denote by $\mathcal S^k_\varepsilon (\Pi_p)$ and $\mathcal B^{k+1}_\varepsilon (\Pi_p)$ the images of the sphere and ball of radius $\e$ in $\Pi_p$ under the exponential map $\exp_p$, $p \in M$.  We can now state our main result. 

\begin{theorem} \label{T1}
If $\Pi_p$ is a nondegenerate critical point of $\calR_{k+1}$, then for all $\e$ sufficiently small, there exists a CMC submanifold $K_\e(\Pi_p)$ which 
is a normal graph over $\mathcal S^k_\varepsilon (\tilde \Pi_{\tilde p})$ by some section with $\calC^{2,\alpha}$ norm bounded by $C \e^2$, and $\mathrm{dist}\,(\tilde \Pi_{\tilde p}, \Pi_p) \leq c\, \varepsilon$. 
\end{theorem}

Our construction of CMC submanifolds generalizes the method introduced in \cite{Pac-Xu}, and can also be carried out in
certain cases when the partial scalar curvature has degenerate critical points, for example when $(M,g)$ is Einstein has or
constant partial scalar curvature. 

\begin{theorem} \label{T2}
There exists $ \e_0 > 0 $ and a smooth function 
$$ \Psi: G_{k+1}(TM) \times (0,\e_0) \longrightarrow \mathbb R, $$
defined in \eqref{defPsi} below, such that if $\e \in (0,\e_0)$ and $\Pi_p $ is a critical point of $\Psi(\cdot, \e) $, then there 
exists an embedded $k$-dimensional 
submanifold $K_\e(\Pi_p)$ with constant mean curvature equal to $k/\e$.  This submanifold is a normal 
graph over a geodesic sphere $\mathcal S^k_\e(\Pi_p)$ with respect to a vector field, the $\mathcal C^{2,\alpha}$ norm of 
which is bounded by $c\e^2$. 
\end{theorem}
The function $\Psi$ is essentially just the associated energy functional restricted to a particular finite dimensional
set of approximately CMC submanifolds.

The outline of this paper is as follows. We first give a more careful description of our proposed definition of constant mean
curvature and its relationship to the associated energy functional. We introduce the linearization and second variation of
this energy, then compute these operators in detail for the round sphere $S^k \subset \RR^{m+1}$.  The construction
of `small' solutions of the CMC problem concentrating around critical points of the function $\Psi$ proceeds in stages.
We construct a family of approximate solutions, then solve the problem up to a finite dimensional defect. This defect
depends on certain parameters in the approximate solution, and in the last step we employ a variational argument 
to choose the parameters appropriately to solve the exact problem.  Certain long technical calculations are relegated
to the appendices. 

\section{Preliminaries} 
In this section we begin by setting notation and recalling some standard formul\ae. This is followed by the introduction of a 
variational notion of constant mean curvature for closed submanifolds of arbitrary codimension. We compute the first and 
second variations of the associated energy functional, and then explain what these look like for round spheres (of arbitrary 
codimension) in $\RR^{m+1}$. 

\subsection{The mean curvature vector}
Let $(M^{m+1},g)$ be a compact smooth Riemannian manifold, and consider smooth, closed $k$-dimensional submanifolds 
$K \subset M$  and $(k+1)$-dimensional submanifolds $Q$ with nonempty boundary $K$, $1 \leq k \leq m$.  We write
$\nabla^\Sigma$ for the connection on any embedded submanifold $\Sigma$, and reserve $\nabla$ for the full Levi-Civita 
connection on $M$.

The second fundamental form of $\Sigma$ is the symmetric bilinear form on $T\Sigma$ taking values in the normal bundle 
$N\Sigma$ defined by
\[
h(X,Y) : = \nabla_X\, Y - \nabla^\Sigma_XY = \pi_{N\Sigma} \, \nabla_X \, Y;
\]
here $ \pi_{N\Sigma} $ is the fibrewise orthogonal projection $T_\Sigma M \to N\Sigma$. The trace of $h$ is a section of $N\Sigma$, 
and is called the mean curvature vector field 
\[
H_\Sigma : = \tr^g \, h \,  = \sum_{i=1}^{\dim \Sigma} h (E_i, E_i),
\]
where $\{E_i\}$ is any orthonormal basis for $T_p\Sigma$. By definition, $\Sigma$ is minimal provided $ H_\Sigma \equiv 0 $.

\subsection{Constant mean curvature in high codimension}  
Let us now specialize to the case where $Q^{k+1} \subset M$ is a smooth, compact submanifold with boundary, with $\del Q = K$. 
The normal bundle $NK$ decomposes as an orthogonal direct sum
\[
NK = NK^{\perp} \oplus NK^{\parallel} \, ,
\]
where $NK^{\parallel} = NK \cap TQ$ has rank $1$ and $NK^{\perp} = N_K (NQ) = NK \cap NQ$ has rank $m-k$.  We shall write $n$ for the inward pointing unit normal to $K$ in $Q$. Thus if $ \Phi \in NK$, then $ \Phi = [\Phi]^\perp + [\Phi]^{\parallel} = [\Phi]^\perp + \phi\, n$ for some scalar function $\phi$. 

\begin{definition}
The closed submanifold $K \subset M$ is said to have {\it constant mean curvature} if $K = \del Q$ where $Q$ is minimal in $M$, $K$ has constant mean curvature in $Q$ and the $Q$-normal component $[H_K]^\perp  \in NK^{\perp}$ vanishes.
\label{def:1}
\end{definition}

A key motivation is that this definition is variational, where the relevant energy is given by
\begin{equation}
\calE_{h_0} (Q) : = {\rm Vol}_k (\partial Q ) - h_0 \,  {\rm Vol}_{k+1} (Q ). 
\label{energy}
\end{equation}
 
\begin{proposition}
The submanifold $K = \del Q$ has constant mean curvature $h_0$ (in the sense of Definition~\ref{def:1}) if and only if 
\[
\left. D\mathcal E_{h_0}\right|_{Q} = 0.
\]
\end{proposition}

The meaning of the differential here is the usual one. Let $\Xi$ be a smooth vector field on $M$ and denote by $\xi_t$ its associated flow.
For $t$ small, write $Q_ t = \xi (t, Q)$ and $K_t := \partial Q_t = \xi (t, K)$. The requirement in the Proposition is then that for any smooth
vector field $\Xi$, 
\[
\left. \frac{d\,}{dt} \calE_{h_0}(Q_t) \right|_{t=0} = 0.
\]

The proof is standard. The classical first variation formula (see Appendix) states that
\[
\left. \frac{d\,}{dt} {\rm Vol}(K_t) \right|_{t=0} =  - \int_{K}  g(H_K  , \Xi)  \, {\rm dvol}_K,
\]
and
\[
\left. \frac{d\,}{dt} \, {\rm Vol} (Q_t ) \, \right|_{t=0}=  - \int_{Q}  g( H_Q , \Xi) \, {\rm dvol}_Q - \int_K  g( n ,  \Xi) \, {\rm dvol}_K \, .
\]
It follow directly from these that 
$$
\left. \frac{d\,}{dt}\right|_{t=0} \calE_{h_0} (Q_t) = 0,
$$ 
for all vector fields $\Xi$ if and only if $H_{K} = h_0 \, n$ and $H_{Q} \equiv 0$, as claimed. 

The definition above coincides with the standard meaning of CMC when $K$ is a hypersurface in $M$ which is the boundary of a region $Q$. In particular, if $K^k \subset \RR^{k+1} \subset \RR^{m+1}$ and $K$ has CMC as a hypersurface in $\RR^{k+1}$, then it has CMC in the sense of Definition~\ref{def:1}. In particular, any round sphere $S^k \subset \RR^{m+1}$ has CMC in this sense. 

\subsection{The Jacobi operator}
Let us now study the differential of the mean curvature operator, which is known as the Jacobi operator. For this subsection, we revert to considering an arbitrary submanifold $\Sigma$, either closed or with boundary, and shall now recall the expression for this operator.

The Jacobi operator $J_\Sigma$ is the differential of the mean curvature vector field with respect to perturbations of $\Sigma$. To describe this more carefully, set $B_\e(N\Sigma) = \{ (q,v) \in T_\Sigma M: |v| < \e\} $ and consider the exponential map $\exp$ from an $\e$-neighborhood of the zero section in $T_\Sigma M$ into $M$. Since $\left. \exp_*\right|_{\{v=0\}} = \mathrm{Id}$, If $\Phi \in \calC^2(\Sigma; T_\Sigma M)$ has $||\Phi||_{\calC^0}$ sufficiently small, then $\Sigma_\Phi:=\{ \exp_q( \Phi(q)): q \in \Sigma\}$ is an embedded submanifold.   We shall denote the family of submanifolds $\Sigma_{s\Phi}$ by $\Sigma_s$, and their mean curvature vector fields by $H_s$.  We also  write $F_s: \Sigma \to \Sigma_s$ for the map $q \mapsto \exp_q (s \Phi(q))$. By definition, 
$$
J_\Sigma (\Phi) = \left. \nabla_{\del/\del s} H_s \right|_{s=0}.
$$

We shall be particularly interested in the case where $\Phi$ is a section of the normal bundle $N\Sigma$.  When $\del \Sigma \neq \emptyset$, we also require that $\Phi = 0$ on $\del \Sigma$.  The operator $\pi_{N\Sigma} \circ J_\Sigma$ will be denoted $J_\Sigma^{N}$. We recall in the appendix the proof of the standard formula 
\begin{equation}
J_\Sigma^{N} = - \Delta^{N}_\Sigma + \Ric_\Sigma^{N} + \frakH^{(2)}_\Sigma,
\label{JacK}
\end{equation}
where $\Delta^{N}_\Sigma$ is the (positive definite) connection Laplacian on sections of $N\Sigma$, 

\smallskip

\begin{center}
\scalebox{0.9}{$ \forall \Phi \in N\Sigma, \quad \Delta_\Sigma^{N} \, \Phi = \sum \limits_{i=1}^{\mathrm{dim}(\Sigma)} \nabla^{ \small{N}}_{E_i} \, \nabla^{N}_{E_i} \, \Phi - \nabla^{N}_{ \huge{ \nabla^{N}_{E_i} E_i } } \, \Phi, \quad \nabla^{N}_X Y = \pi_{N_\Sigma} \circ \nabla_X Y $}
\end{center}


\medskip

and the other two terms are the following symmetric endomorphisms of $N\Sigma$:
\begin{itemize}
\item[(i)] 
The orthogonal projection $ \Ric_{\Sigma}^N = \pi_{N_\Sigma} \circ \Ric_\Sigma $ of the partial Ricci curvature $\Ric_\Sigma$,  defined by 
\begin{equation}
\begin{split}
\langle \Ric_\Sigma \, X , Y \rangle &:=  - \tr^g \, \langle R(\cdot , X) \cdot, Y)   \\ 
& =  - \sum_{i=1}^{\dim \Sigma} \langle R(E_i , X) E_i, Y \rangle \quad \mbox{for all} \quad X, Y \in TM  
\end{split}
\label{parRic}
\end{equation}
note that the curvature tensor appearing on the right is the one on all of $M$, and is not the curvature tensor for $\Sigma$; 
\item[(ii)] the square of the shape operator, defined by
\begin{equation}
\frakH_\Sigma^{(2)}(X) := \sum_{i,j=1}^{\dim \Sigma}  \langle h(E_i, E_j), X \rangle h (E_i, E_j), \quad \mbox{for all} \quad X \in TM
\label{sqshape}
\end{equation}
\end{itemize}

In general, $J_\Sigma(\Phi) \neq J^{N}_\Sigma (\Phi)$ since $J_\Sigma(\Phi)$ has a nontrivial component $J_\Sigma^{T}(\Phi)$ which is parallel to $\Sigma$; as we show later, that part is canceled in our final formula so we do not need to make it explicit. Note, however, that $J_\Sigma^{T}(\Phi)$ vanishes when $\Sigma$ is minimal. Indeed, writing the mean curvature vector field to $\Sigma_{s\Phi}$ in the form
$$
H_s = \sum_\nu \langle H_s, N_\nu(s)\rangle N_\nu(s),
$$ 
where $N_\nu(s)$, $\nu = \dim \Sigma +1, \ldots , m+1$ is a local orthonormal frame for $N\Sigma_{s\Phi}$ we find
\begin{multline*}
[J_\Sigma(\Phi)]^{T}  =  \sum_\nu  \Big( \left( \langle \left. \nabla_{\del/\del s} H_s \right|_{s=0}  , N_\nu(0) \rangle  + \langle H_\Sigma,  \left. \nabla_{\del/\del s} \right|_{s=0} N_\nu(s) \rangle \right) N_\nu(0) \\[2mm]  
\langle H_\Sigma, N_\nu(0) \rangle \left. \nabla_{\del/\del s} \right|_{s=0} N_\nu \Big)^{T} = \sum_\nu \langle H_\Sigma, N_\nu \rangle \left[ \left. \nabla_{\del/ \del s} N_\nu(s) \right|_{s=0} \right]^{T},
\end{multline*}
and if $H_\Sigma = 0$, we have $ J_{\Sigma}^{T} = 0 $.
 
\subsection{The second variation of $\calE_{h_0}$} 
We set 
\[
\calC^{2, \alpha}_0 (NQ) : = \{V \in \calC^{2,\alpha}(NQ): \left. V \right|_K = 0\}. 
\]
With this notation in mind, we have the:
 \begin{definition} \label{nondeg}
The minimal submanifold $Q$ is nondegenerate if 
\[
J_Q  : \calC^{2, \alpha}_0 (NQ)   \longrightarrow \calC^{0, \alpha} (NQ) , 
\]
is invertible. 
\end{definition}

\begin{lemma} \label{MinSubman} If $Q$ is nondegenerate, then there is a smooth mapping $\Phi \mapsto Q_\Phi$ from a neighbourhood of $0$ in $\calC^{2,\alpha}(NK)$ into the space of $(k+1)$-dimensional minimal submanifolds of $M$ with $\calC^{2,\alpha}$ boundary, such that $Q_0$ is the initial submanifold $Q$ and $\del Q_\Phi = K_\Phi$. 
\label{defQ}
\end{lemma}
\begin{proof} 
Fix a continuous linear extension operator   
\[
\calC^{2,\alpha}(NK) \ni \Phi \mapsto V_\Phi \in \calC^{2,\alpha}(T_Q M).
\]
Thus $V_\Phi$ is a vector field along $Q$ which restricts to $\Phi$ on $K$. Without loss of generality, we can assume that $V_\Phi \in TQ$ 
if $[\Phi]^\perp =0$ and $V_\Phi \in NQ$ when $[\Phi]^\| =0$. Next, let $W$ be a $\calC^{2,\alpha}$ section of $NQ$ which vanishes on 
$K$. If both $||\Phi||_{2,\alpha}$ and $||W||_{2,\alpha}$ are sufficiently small, then $\exp_Q (V_\Phi + W)$ is an embedded $\calC^{2,\alpha}$ 
submanifold $Q_U$, $U = V_\Phi + W$, with boundary $K_\Phi : = \partial Q_U$.  Denoting its mean curvature vector by $ H(\Phi,W)$, 
then 
\[
\left. D_WH\right|_{(0,0)} (W) = J_Q W.
\]

Since $Q$ is minimal, $\left. D_WH\right|_{(0,0)} (W)$ takes values in $NQ$, whereas $H(\Phi,W) \in N{Q_U} \subset T_{Q_U}M$, so 
we cannot directly apply the implicit function theorem. To remedy this, first let $\widetilde{H}(\Phi,W)$ be the parallel transport 
of $H(\Phi,W)$ along the geodesic $s \mapsto \exp_q(s U(q))$, from $s=1$ to $s=0$. Parallel transport preserves regularity 
(this reduces to the standard result on smooth dependence on initial conditions for the solutions of a family of ODE's), 
so $\widetilde{H}(\Phi,W)$ is a $\calC^{0,\alpha}$ section of $T_Q M$. Now define 
$$
\widehat{H}(\Phi,W) := \pi_{NQ} \circ \widetilde{H}(\Phi,W),
$$ 
where $\pi_{NQ} : T_Q M \to NQ$ is the orthogonal projection. Since $H(\Phi,W) \in N_{Q_U} M$ and since $||U||_{\calC^1}$ is small, 
$\widetilde{H}(\Phi,W)$ lies in the nullspace of $\pi$ at any $q \in Q$ if and only if it actually vanishes. Thus it is enough to look 
for solutions of $\widehat{H}(\Phi,W) = 0$. Notice that $D_W \widehat{H}|_{(0,0)} = J_Q$. We can now apply the implicit function 
theorem to conclude the existence of a $\calC^{2,\alpha}$ map $\Phi \mapsto W(\Phi)$ such that $H(\Phi, W(\Phi)) =
\widehat{H}(\Phi,W(\Phi)) \equiv 0$ for all small $\Phi$.
\end{proof}

We henceforth denote by $Q_\Phi$ the minimal submanifold $ \exp_Q \left( V_\Phi + W(\Phi) \right)$. Observe that when $[\Phi]^\perp =0$,
 the submanifold parametrized by $\exp_Q (V_\phi)$ is $\mathcal O ( \| \Phi\|_{\mathcal C^{2, \alpha}}^2)$ close to $Q_\Phi$; this is 
easy to check when $\Phi : = \phi \, n$ where $\phi$ is small. Therefore, in this `tangential' case, we conclude that 
\[
U_\Phi  = V_{\Phi}  + \mathcal O ( \| \Phi\|_{\mathcal C^{2, \alpha}}^2).
\]
Next, when $[\Phi]^\| =0$, we define $Z_\Phi$ as the solution of 
\[
J_Q Z_{\Phi} = 0, \qquad \left. Z_{\Phi} \right|_K = \Phi , 
\]
and it is easy to check that the submanifold parametrized by $\exp_Q (Z_\phi)$ is also 
$\mathcal O ( \| \Phi\|_{\mathcal C^{2, \alpha}}^2)$ close to $Q_\Phi$. We summarize all this in the
\begin{lemma}
When $\| \Phi\|_{\mathcal C^{2, \alpha}} $ is small, we have the decomposition 
\[
U_\Phi  = V_{[\Phi]^\|}  + Z_{[\Phi]^\perp} + \mathcal O ( \| \Phi\|_{\mathcal C^{2, \alpha}}^2),
\]
where $Z_{[\Phi]^\perp}$ is the solution of 
\[
J_Q Z_{[\Phi]^\perp} = 0, \qquad \left. Z_{[\Phi]^\perp} \right|_K = [\Phi]^\perp .
\]
\end{lemma}

Now consider the energy $\calE_{h_0}$ along a one-parameter family $s \mapsto Q_s := Q_{s\Phi}$ of minimal submanifolds 
with boundaries $K_s : = \partial Q_s = K_{s\Phi}$.  By the formul\ae\ of the last subsection, 
\[
\frac{d\,}{ds} \calE_{h_0}(Q_s) = - \int_{K_s} g( H_s - h_0 \, n_s, \partial/\partial s) \, {\rm dvol}_{K_s} ,
\]
where $H_s$ is the mean curvature of $K_s$ and $n_s$ is the inward pointing unit normal to $K_s$ in $Q_s$.  Note that this first variation 
of energy is localized to the boundary; the interior terms vanish because of the minimality of the $Q_s$. Our task is to compute 
\[
\left. \frac{d^2\,}{ds^2} \calE_{h_0}(Q_s) \right|_{s=0} ,
\]
when $Q$ is critical for $\mathcal E_{h_0}$. 

Parametrize both $K_{s}$ and $Q_{s}$ by $y \mapsto F_s(y) := \exp_y( U_{s\Phi}(y))$ (with $ y \in K $ or $ y \in Q $, respectively).  
As before, choose a smooth local orthonormal frame $E_\alpha$ for $TK$, so that $(F_s)_* E_\alpha = E_\alpha(s)$ is a local 
(non-orthonormal) frame for $TK_{s\Phi}$. We then
include $n(s)$, the unit inward normal to $K_{s}$ in $Q_s$. Moreover, we extend $n(s)$ to a vector $\bar n(s) \in TQ_s$ so that it 
satisfies $\nabla^{Q_s}_{\bar n(s)} \bar n(s) = 0$. We supplement this to a complete local frame for $T_{Q_s}M$ (at least near points 
of $K_s$) by adding a local orthonormal frame $N_\mu(s) \in NQ_s$. 
Here we let the indices $\alpha, \beta, \ldots$ run from $1$ to $k$ while $\mu, \nu, \ldots $ run from $k+1$ to $m+1$ . 

\begin{notation} 
Set $ \mathcal H(s) = H(K_s) - h_0 \, H(Q_s) $, where $ h_0 = H_K $. We also write
$$ L_Q =  \left. \nabla_{\del / \del s } \mathcal H_s \right|_{s=0} $$ 
\end{notation}
Note that we can decompose $\mathcal H'(0)$ into $\mathcal H'(0)^{N_K} + \calH'(0)^{T_K}$, its components perpendicular 
and parallel to $K$. Since $\calH(s) \perp K_s$, we have that $\langle \calH(s) , E_\alpha(s) \rangle = 0$, so 
\[
\langle \calH'(0), E_\alpha \rangle + \langle \calH(0), E_\alpha'(0) \rangle = 0.
\]
Since $\calH(0) = 0$, we obtain $ [L_Q]^{T_K} = 0$. 

\medskip

Next decompose $\Phi = [\Phi]^{\perp} + \phi \, n$ into parts perpendicular and parallel to $Q$ (along $K$).  Noting that we can choose the vector field $V_\Phi$ extending $\Phi$ in Lemma~\ref{defQ} so that its component tangent to $Q$ lies in the span of $n$, there is a similar decomposition $ U_\Phi =  [U_\Phi]^\perp + u_\phi \, \bar n(s)$ for the vector field $U_{\Phi}$ constructed in that Lemma, locally near $K_\Phi$; note that $\left. [U_\Phi]^\perp \right|_K = [\Phi]^\perp$ and $\left. u_\phi \right|_{K} = \phi$.  

\medskip

To see that $E_\alpha'(0) = \nabla_{E_\alpha} \Phi$, choose a curve $c(t)$ in $K$ with $c(0) = p$, $c'(0) = E_\alpha$ and define 
$G(t,s) = \exp_{c(t)} (s \Phi( c(t)))$; we then obtain that
\[
\left. \nabla_{\del/\del s} E_\alpha\right|_{s=0} = \left. \nabla_{\del/ \del s} \nabla_{\del / \del t} \right|_{s=t=0}  G(t,s) = \left. \nabla_{\del/\del t} \Phi(c(t))\right|_{t=0} = \nabla_{E_\alpha} \Phi,
\]
as claimed. To compute $n'(0)$, observe that $(F_s)_*(n(0))$ is always tangent to $Q_s$ and transverse, but not necessarily a unit normal, to $K_s$.
We can adjust it, using the Gram-Schmidt process, to get that 
\[
n(s) = \left( (F_s)_*(n(0)) - \sum c_\alpha E_\alpha(s)\right)/ \left| ((F_s)_*(n(0)) - \sum c_\alpha E_\alpha(s)\right|,
\]
where 
\[
c_\alpha(s) = \langle E_\alpha(s), (F_s)_* n(0) \rangle/ |E_\alpha(s)|^2.
\]
Arguing as before, take a curve $ d(t) $ in $Q$ such that $ d(0) = p$ and $d'(0) = n$ and define 
$\tilde G(t,s) = \exp_{d(t)}(U_{s\Phi}(d(t))) $. Note that $ U_{s\Phi} = s ( V_{[\Phi]^\parallel} + Z_{[\Phi]^\bot}) + 
\mathcal O(s^2 \| \Phi \|^2_{\mathcal C^{2,\alpha}} ) $. We get 
$$ \left. \nabla_{\del/\del s} (F_s)_* n(0) \right|_{s=0} = \left. \nabla_{\del / \del s} \nabla_{\del / \del t} \tilde G(t,s) \right|_{t=s=0} = \nabla_{n} (V_{[\Phi]^\parallel} + Z_{[\Phi]^\perp}) $$ 
and since $c_\alpha(0) = 0$, we obtain
\[
[n'(0)]^{\perp} = \left. \left[\nabla_n V_{[\Phi]^\parallel} + \nabla_n Z_{[\Phi]^\bot} \right] \right|_K^\perp = \left.  \left[ \nabla_n^\perp Z_{\Phi^\perp} + \phi \, \nabla^\perp_n \bar n \right] \right|_K.
\]
Finally, the component $ [n'(0)]^\parallel = 0 $.
Combining these calculations gives the 
\begin{proposition}
If $Q$ is critical for $\calE_{h_0}$, then 
\[ 
L_Q \, \Phi = J_K^{N_K} \Phi - h_0 \, D_Q \Phi, 
\]
where 
\[
D_Q \Phi = \left. \left[ \nabla_n^\perp Z_\Phi + \phi \, \nabla^\perp_n \bar n \right] \right|_K
\]
\end{proposition}
\
\subsection{The linearization at $K = S^k$}
We conclude this section by discussing the precise form of this linearization, and its nullspace, when 
\[
K = S^k \times \{0\} \subset Q = B^{k+1} \times \{0\} \subset \RR^{m+1},
\]
since this is our basic model later.  It is easy to see that $B^{k+1}$ is critical for $\calE_{k}$. 

The unit inward normal to $S^k$ in $B^{k+1}$ is $ n_{S^k}(\Theta) = - \Theta $.  If $ \Phi \in \calC^{2,\alpha}(NS^k)$, then
\[
\Phi = [\Phi]^\perp - \phi \, \Theta,
\]
where the first term on the right is perpendicular to $B^{k+1}$. The operator $J_{S^k}^{N}$ acts on these two components 
separately, via $J_{S^k}^{\perp}$ and $J_{S^k}^\parallel$, respectively. 

The first of these operators acts on sections of the trivial bundle of rank $m-k$.  Obviously, $\Ric^N_{S^k} = 0$, cf.\ \eqref{parRic}, 
and $(\frakH^{(2)}_{S^k})^\perp = 0$ as well, so 
\[
J_{S^k}^\perp = \Delta_{S^k} 
\]
acting on $(m-k)$-tuples of functions. Its eigenvalues are $\ell( k + \ell-2)$. The operator $D_{B^{k+1}}$ also acts on sections 
of the trivial bundle $\left. N B^{k+1}\right|_{S^k}$. In fact, since $J_{B^{k+1}} = \Delta_{B^{k+1}}$, this operator is simply the standard 
Dirichlet-to-Neumann operator for the Laplacian (acting on $\RR^{m-k}$-valued functions). Its eigenfunctions are the restrictions 
to $r=1$ of the homogeneous harmonic 
polynomials $P(x)$, $x = r \Theta $, $ \Theta \in S^k$.  If $P$ is homogeneous of order $\ell$, then $P(x) = r^\ell P(\Theta)$,
so $D_{B^{k+1}} P(\Theta) = -\ell P(\Theta)$ (recall we are using the inward-pointing normal). Combining these two
operators, we see that $\Delta_{S^k} - k D_{B^{k+1}}$ has eigenvalues $-\ell(k + \ell -1) + k\ell = - \ell(\ell-1)$, hence
\[
\left( J_{S^k}^\perp - k D_{B^{k+1}} \right) [\Phi]^\perp = 0 \Rightarrow [\Phi]^\perp \in \mathrm{span}\,\{ (a_\mu + b_\mu x_\mu) E_\mu \},
\]
where $E_\mu$, $\mu = k+2, \ldots, m+1$ is an orthonormal basis for $N B^{k+1} = \RR^{m-k}$. 

The remaining part is 
\[
J_{S^k}^{\parallel} =  \Delta_{S^k} + k,
\]
since $\Ric_{S^k} = 0$ and $\frakH^{(2)}_{S^k} = k \, \mathrm{Id}$. Thus
\[
J_{S^k}^\parallel  ( \phi \, \Theta ) = J_{S^k}^\parallel (\phi) \, \Theta = 0 \Rightarrow  \phi \in \mathrm{span}\,\{x_1, \ldots, x_{k+1}\}.
\]

We have now shown that the nullspace $\calK$ of $L_{B^{k+1}}$ splits as $\calK^\perp \oplus \calK^{\parallel}$.
The first of these summands is comprised by infinitesimal translations in $\RR^{m-k}$ and infinitesimal rotations in the $\alpha \mu$ 
planes (now $\alpha \leq k+1$); the second summand corresponds to infinitesimal translations in $\RR^{k+1}$. 

\section{Construction of constant mean curvature submanifolds}
We now turn to the main task of this paper, which is to construct small constant mean curvature submanifolds 
concentrated near the critical points of $\calR_{k+1}$.    The first step is to define a family of approximate solutions,
i.e., a family of pair $(Q_\e, K_\e)$ where $Q_\e$ is minimal and has nearly CMC boundary. We then use a variational 
argument to perturb this to a minimal submanifold with exactly CMC boundary.

\subsection{Approximate solutions} 
We adopt all the notation used earlier. Thus we fix $\Pi_p \in G_{k+1}(TM)$ and an orthonormal basis $ E_i $, 
$ 1 \leq i \leq m+1$ of $T_pM$, where $E_a$, $1 \leq a \leq k+1$ span $\Pi_p$ and $E_\mu$, $\mu > k+1$, 
span $\Pi_p^\perp$. This induces a Riemann normal coordinate system $(x^1, \ldots, x^{m+1})$ near $p$, and it is standard that 
\begin{equation}
g_{ij}(x) = g(\del_{x^i}, \del_{x^j}) = \delta_{ij} + \frac13 \sum_{k,\ell} (R_p)_{ikj\ell} \, x^k x^\ell + \mathcal O(|x|^3), 
\label{expg}
\end{equation}
where $\delta$ is the Euclidean metric. 

\subsubsection{Rescaling}
In terms of the map $F_\e: T_p M \to M$, $F_\e(y) = \exp_p ( \e y)$, used earlier, define the metric 
$$ g_\e = \e^{-2} F_\e^* g $$ 
on $T_pM$, or equivalently, work in the rescaled coordinates $y^j = x^j/\e$.
In either case, 
\begin{equation}
g_\e =  |dy|^2 + \e^2 h_\e(y, dy),
\label{estge}
\end{equation}
where $h_\e$ is family of smooth symmetric two-tensors depending smoothly on $\e \in [0,\e_0]$. 
The mean curvature vectors $H^g$ and $H^{g_\e}$ with respect to $g$ and $g_\e$ satisfy
\[
\e^2 \, H^g = (F_\e)_* \, H^{g_\e}, \quad \mbox{and} \qquad \| H^{g_\varepsilon}\|_{g_\varepsilon} = \varepsilon \, \|H^{g}\|_g. 
\]

Let $B^{k+1} = B^{k+1}(\Pi_p) \subset \Pi_p$ be the unit ball and $S^{k+1} = S^{k+1}(\Pi_p) = \del B^{k+1}$, and denote their
images under $F_\e$ by $\calB_\e$ and $\calS_\e$. 
These have parametrizations
\[
S^{k+1} \ni \Theta \longmapsto \exp_p^g ( \e \, \Theta), \quad B^{k+1} \ni y \longmapsto \exp_p^g 
( \e \, \sum_{a=1}^{k+1} y^a E^a ).
\]

In the lemmas \eqref{ball} and \eqref{sphere} below we give the expansion of the mean curvature of $\mathcal B_\e$ and $\mathcal S_\e$ in terms of $\e$. To this end we indroduce two supplementary curvature invariants which are restrictions of the Ricci curvature of the ambient manifold $M$:

\begin{equation*}
\begin{split}
& \mathcal Ric_{k+1}(\Pi_p)(v_1,v_2) = - \sum_{i=1}^{k+1} R_p(E_i,v_1,E_i,v_2), \qquad v_1, v_2 \in \Pi_p \\
& \mathcal Ric_{k+1}^\bot(\Pi_p)(v,N) = - \sum_{i=1}^{k+1} R_p(E_i,v,E_i,N), \qquad v \in \Pi_p, N \in \Pi_p^\bot.
\end{split}
\end{equation*}

Note that 
\[
\mathcal Ric_{k+1}^\bot(\Pi_p) = \left[ \Ric^N_{\mathcal B_\e} \right]_p. 
\]
\begin{lemma} \label{ball}
 We have
 $$ H^g(\mathcal B_\e)(y) = \sum_{\mu=k+1}^{m+1} \left( \dfrac{2 \, \e}{3} \, \mathcal Ric^\bot_{k+1}(\Pi_p)(y,E_\mu) + \mathcal O(\e^2) \right) \, \mathcal N_\mu $$
 where $ \mathcal N_\mu $, $ k+1 \leq \mu \leq m+1 $ is an orthonormal basis of $N \mathcal B_\e$.
\end{lemma}

\begin{remark}
Here and below, we write $\calO(\e^k)$ for a function with $\calC^{0,\alpha}$ norm bounded by $C \e^k$. 
\end{remark}
\begin{proof}
Recall that
$$ H^g(\mathcal B_\e) = \dfrac{1}{\e^2} \, (F_\e)_* \, H^{g_\e}(B^{k+1}) $$
We denote $ \mathcal N_\mu^\e$, $ k+1 < \mu < m+1 $ the orthonormal basis of the normal bundle of $ B^{k+1} \in T_pM $ with respect to the metric $g_\e$ obtained by applying the Gram-Schmidt process to the vectors $E_\mu$, $ k+1 \leq \mu \leq m+1 $. Remark that
$$ g_\e( \mathcal N_\mu,E_\nu) = \delta_{\mu\nu} + \mathcal O(\e^2) $$
We denote $ \mathcal N_\mu = \e \, \mathcal N_\mu^\e $
the orthonormal basis of the normal bundle to $B^{k+1}$ with respect to the metric $ (F_\e)^* g $. We identify $ \mathcal N_\mu $ with $ (F_\e)_* \, \mathcal N_\mu $; these last vector fields form an orthonormal basis of $ N \mathcal B_\e $ with respect to the metric $g$.
The Christoffel symbols corresponding to the metric $g_\e$ are:
\begin{align*}
 (\Gamma^{g_\e})^\ell_{ij}(y) & = \frac12 g_\e^{kq}\left( \del_{y^j} (g_\e)_{iq} + \del_{y^i} (g_\e)_{jq} - \del_{y^q} (g_\e)_{ij} \right) \\[3mm]
 & =  \delta^{q \ell} \,  \frac{ \e^2}{6} \, y^p \, \left( R_{i j q p} + R_{ i p q j}  + R_{j i q p} + R_{j p q i} - R_{iqjp} - R_{ipjq} \right)
 + \calO(\e^3) \\[3mm]
 & = - \frac{\e^2}{3} \left( R_{ipj\ell} + R_{i\ell jp} \right) \, y^p + \mathcal O(\e^3)
\end{align*}
whence
$$ g_\e( \nabla^{g_\e}_{\partial_{y^a}} \partial_{y^b}, \mathcal N_\mu^\e ) = (\Gamma^{g_\e})_{ab}^\mu + \mathcal O(\e^4) $$
\medskip
Taking the trace in the indices $ a,b = 1, \ldots, k+1 $ with respect to $g_\e$ gives the result. 
\end{proof}

\begin{lemma} \label{sphere}
We have
 \begin{multline*} 
  H^g (\mathcal S^k_\varepsilon ) = \left( \frac{k}{\e} - \frac{\e}{3} \, \mathcal Ric_{k+1} (\Pi_p)(\Theta,\Theta) + \mathcal O(\e^2) \right) \, n_{\mathcal S} \\[3mm] 
  + \sum_{\mu=k+1}^{m+1} \left( \frac{ 2 \,\e}{3} \mathcal Ric_{k+1}^\bot(\Pi_p)(\Theta,E_\mu) + \mathcal O(\e^2) \right) \, \mathcal N_\mu
 \end{multline*}
 where $n_{\mathcal S}$ is a unit normal vector field to $\mathcal S^k_\e$ in $\mathcal B^{k+1}_\e$ with respect to the metric $g$. 
 \end{lemma}

\begin{proof}
The proof is similar to that of the previous lemma, but with several changes.
 
\medskip

Let $u^1, \ldots, u^k \mapsto \Theta(u^1,\ldots,u^k)$ be a local parametrization of $S^k \subset \Pi_p$. The tangent bundle $TS^k$ is spanned by the vector fields $ \Theta_\alpha = \partial_{u^\alpha} \Theta $. As before, we have
$$ H^g(\mathcal S^k_\e) = \dfrac{1}{\e^2} \, (F_\e)_* \, H^g_\e(S^k) $$
By the Gauss lemma,
$$ g \left( (F_\e)_* \Theta_\alpha, (F_\varepsilon)_*\Theta \right) ( F_\e (\Theta) ) = g_p(\Theta_\alpha, \Theta ) = 0 $$
and
$$ g \left( (F_\e)_* E_\mu, (F_\e)_* \Theta \right) ( F_\e (\Theta) ) = g_p( E_\mu, \Theta ) = 0 $$
this yields
$$ g(\mathcal N_\mu, (F_\e)_* \Theta) = 0 \quad \mbox{and} \quad g_\e(\mathcal N_\mu^\e,\Theta) = 0 $$

Finally we put $ n_{\mathcal S}:= - (F_\e)_* \Theta $. We have 
$$ 
\nabla^{g_\e}_{\partial_{u^\alpha}} \partial_{u^\beta} = \partial_{u^\alpha} \, \partial_{u^\beta} \Theta + (\Gamma^{g_\e})_{ij}^\ell 
(\Theta_\alpha)^i (\Theta_\beta)^j \, E_\ell 
$$
$\alpha, \beta = 1, \ldots, k$, $i,j,\ell = 1, \ldots, m+1$. 

\medskip

The vector field $ \partial_{u^\alpha} \, \partial_{u^\beta} \Theta $ is tangent to $B^{k+1}(\Theta)$, so 
$$ 
g_\e \left( \nabla^{g_\e}_{\partial_{u^\alpha}} \partial_{u^\beta}, \mathcal N^\e_\mu \right) = (\Gamma^{g_\e})_{ab}^\mu \, (\Theta_\alpha)^a \, (\Theta_\beta)^b + \mathcal O(\e^3).
$$
Taking trace in the indices $\alpha, \beta$ with respect to the metric induced on $S^k$ from $g_\e$ we get
$$ 
g_\e( H^{g_\e}(S^k), \mathcal N_\mu) = \frac{ 2 \, \e^2}{3} \, \mathcal Ric_{k+1}^\perp(\Pi_p)(\Theta, E_\mu) + \mathcal O(\e^3).
$$

\medskip

In order to find $ [H^{g_\e}( S^k(\Pi_p) )]^{||} $, recall the standard fact that if $\Sigma \subset M$ is an oriented 
hypersurface with unit inward pointing normal $N_\Sigma$, and if $\Sigma_z$ is the family of hypersurfaces defined by
\[
\Sigma \times \mathbb R (q,z) \mapsto \exp_q( z N_\Sigma(q) ) \in \Sigma_z,
\]
with induced metric $g_z$, then 
\[
\left|H_\Sigma\right| = - \frac{d}{dz} \ \log \sqrt{ \det g_z }.
\]

In our case, considering $S^k(\Pi_p) \subset B^{k+1}(\Pi_p)$ with metric $g_\varepsilon$, let $g_{\varepsilon z}$ be the induced metrics on the
Euclidean sphere of radius $1-z$. Then 
\[
\det{ g_{\varepsilon z} } = (1-z)^{2k} \det{ g^S } \left( 1 - \frac{\varepsilon^2(1-z)^2}{3} { \mathcal Ric_{S^k} } (\Pi_p) (\vec \Theta, \vec \Theta) 
+ \mathcal O(\varepsilon^3) \right), 
\]
where $ g^S $ is the standard spherical metric on $S^k(\Pi_p)$.  From this we deduce that
\[
g_\e \left( H^{g_\e} ( S^k ), - \Theta \right) = \frac{k}{\e} - \frac{\e}{3} \, \mathcal Ric_{k+1} (\Pi_p)( \Theta, \Theta) + \mathcal O(\e^2).
\] 
this completes the proof.
\end{proof}

\begin{proposition} \label{expVol}
Fix $ \Pi_p \in G_{k+1}(TM)$. Then for $\e > 0$ small enough, there exists a {\it minimal} submanifold $Q_{\e} (\Pi_p)$ which is a 
small perturbation of $\mathcal B^{k+1}_\e(\Pi_p)$, whose boundary $ K_\e(\Pi_p) = \partial Q_\e(\Pi_p) $ is a normal graph over 
$ S^k_\e(\Pi_p) $ and whose mean curvature vector field satisfies
\begin{equation} \label{MeanCurvNullTerms}
 H^g (K_\e (\Pi_p) )  - \dfrac{k}{\e} \, n_K = g_p( \vec a, \Theta ) \, n_K + \sum_{ \mu = k+1 }^{m+1} \left( g_p( \vec c_\mu , \Theta )  
+ d_\mu \right) N_\mu
\end{equation}
for some constant vectors $ \vec a = \vec a(\e,\Pi_p)$, $\vec c_\mu = \vec c_\mu(\e,\Pi_p) \in \Pi_p $ and constants 
$ d_\mu = d_\mu(\e,\Pi_p) \in \RR $.  Here $n_K$ is a normal vector field to $K_\e(\Pi_p)$ in $Q_\e(\Pi_p)$ and $N_\mu$ 
form an orthonormal basis of $ \left[ NK_\e (\Pi_p) \right]^\bot$.
\end{proposition}

\medskip
\begin{proof}
Take a vector field $ \Phi \in \mathcal C^{2,\alpha} (T_p M) $ defined along the unit sphere $S^k(\Pi_p)$, such that
$$ \Phi(\Theta) = - \phi(\Theta) \, \Theta + \sum_{\mu=k+1}^{m+1} \Phi^\mu(\Theta) \, E_\mu, $$
and write
$$ 
S^k_\Phi = \left\{ \Theta + \Phi(\Theta), \ \Theta \in S^k \right\}.
$$
\newline Then there exists a submanifold $B^{k+1}_{\e,\Phi}$ such that $\partial B^{k+1}_{\e,\Phi} = S^k_\Phi$ and which is minimal 
with respect to $g_\e$. The proof of this fact is almost the same as the proof of the Lemma \eqref{MinSubman}; the only difference 
is that we use a "perturbed" metric and the starting submanifold is no longer minimal. Let $V_\Phi$ be a linear extension of $\Phi$ 
in $B^{k+1}$ and take 
$$ 
W \in \mathcal C^{2,\alpha}(T_pM), \quad W = \sum_{\mu = k+1}^{m+1} W^\mu \, E_\mu, \quad \left. W  \right|_{S^k} = 0.
$$
We let $ H(\e, \Phi, W) $ denote the mean curvature with respect to the metric $g_\e$ of the submanifold 
$ \{ U(y) = V_\Phi(y) + W(y), \ y \in B^{k+1} \} $. Note that $ H(0,0,0) = 0 $ and 
$$ 
\left. D_3 H \right|_{(0,0,0)} = J_{B^{k+1}} = \Delta_{B^{k+1}}. 
$$
We can then apply the implicit function theorem to $\hat H(\e,\Phi,W) = \pi\circ H(\e,\Phi,W)$, where $\pi$ is the orthogonal
projection onto the vertical subspace of $T_pM$, which is spanned by $E_\mu$, $ k+1 \geq \mu \leq m+1  $. 
Then for $\e$ and $ \| \Phi \|_{\mathcal C^{2,\alpha}} $ small enough, there exists a mapping $ (\e, \Phi) \mapsto W(\e,\Phi) $ 
such that 
$$ 
\hat H(\e,\Phi,W(\e,\Phi)) = 0 \quad \mbox{and} \quad H(\e,\Phi,W(\e,\Phi)) = 0.
$$
Moreover, 
\[
U_{\e,\Phi} = V_\Phi + W(\e,\Phi) = V_\phi + Z_\Phi + W_\e + \mathcal O(\| \e^3 \|) + \mathcal O( \e^2 \| \, \Phi \| ) + \mathcal O(\| \Phi^2 \|)
\]
where $ V_\phi(y) = - \phi(y/\| y \|) \, y $, the vector field $Z_\Phi$ is the harmonic extansion of $\Phi$ in $B^{k+1}$ and $W_\e$ satisfies
 $$ \Delta_{B^{k+1}} \, W_\e^\mu = \dfrac{ 2 \, \e^2}{3} \, \mathcal Ric_{k+1}^\bot(\Pi_p)(y,E_\mu), \quad W_\e = 0 \quad \mbox{on} \quad S^k $$
 
\begin{remark}
A simple calculation shows that
\[
W_\e(y) = - \, \dfrac{\e^2}{3} \, \dfrac{1}{k+3} \,(1-|y|^2) \sum \limits_{\mu=k+1}^{m+1} \mathcal Ric_{k+1}^\bot(\Pi_p)(y,E_\mu) \, E_\mu.
\]
\end{remark}

\medskip
For the second step, we calculate the mean curvature of $ S^k_\Phi $ with respect to the metric $g_\e$. First note that the vector fields
$$ 
\tau_\alpha = (1 - \phi) \, \Theta_\alpha - \partial_{u_\alpha} \phi \, \Theta + \sum_{\mu=k+1}^{m+1} \partial_{u_\alpha} \Phi \, E_\mu
$$
locally frame $TS^k_\Phi $, while 
$$ 
\Theta_\Phi = \Theta + \dfrac{1}{1-\phi} \, \nabla_{S^k} \phi, \quad \mbox{and} \quad (E_\mu)_\Phi = E_\mu - \dfrac{1}{1-\phi} \, 
\nabla_{S^k} \Phi^\mu
$$
are a local basis for the normal bundle of $S^k_\Phi$ with respect to the Euclidean metric. Applying the Gram-Schmidt process 
with respect to the metric $g_\e$ to these local frames yields the unit normal to $S^k_\Phi$ in $ B^{k+1}_{\e,\Phi}$, which we
denote $ n_\Phi^\e$, and the orthonormal frame $(\mathcal N_\Phi)_\mu^\e $ for the normal bundle of $B^{k+1}_{\e,\Phi}$ along 
$S^k_\Phi$ with respect to $g_\e$. These calculations show that 
\begin{multline*} 
 \left\langle n^\e_\Phi, - \Theta_\Phi / | \Theta_\Phi |_{g_{eucl}} \right\rangle_{g_\e} = 1 + \mathcal O(\e^2) \\[3mm] 
 \left\langle (\mathcal N_\mu)_\Phi, (E_\mu)_\Phi^\e / | (E_\mu)_\Phi |_{g_{eucl}} \right\rangle_{g_\e} = 1 + \mathcal O(\e^2),
\end{multline*}
and $ n^\e_0 = - \Theta $ and $ (\mathcal N_\mu)^\e_0 = \mathcal N_\mu^\e $. We can then write
\begin{multline*}
H^{g_\e}(S^k_\Phi) - k \, n_\Phi  \\ 
= \left( g_\e \left( H^{g_\e}(S^k_\Phi), n_\Phi^\e \right) - k \right) \, n_\Phi^\e + \sum \limits_{\mu=k+1}^{m+1} 
g_\e \left( H^{g_\e}(S^k_\Phi), ( \mathcal N_\Phi)_\mu^\e \right) \, ( \mathcal N_\Phi)_\mu^\e.
\end{multline*}
\begin{notation}
We let $L_{\Pi_p} (\Phi) $ denote any second order linear differential operator acting on $\Phi$. The coefficients of 
$L_{\Pi_p}(\Phi)$ may depend on $ \Pi_p \in G_{k+1}(TM) $ and $ \e \in (0,1) $, but for all $ j \in \mathbb N $ there exists a constant 
$ C_j > 0 $ independent of $\Pi_p$ and $\e$ such that
$$ 
\| L_{\Pi_p} (\Phi) \|_{\mathcal C^{j,\alpha}(S^k)} \leq C_j \, \| \Phi \|_{\mathcal C^{j+2,\alpha}(NS^k)}.
$$
Similarly, for $\ell \in \mathbb N$, $Q^\ell_{\Pi_p}(\Phi) $ denotes some nonlinear operator in $\Phi$, depending also on $\Pi_p$ 
and $\e$, such that $Q_{\Pi_p}^\ell(0) = 0 $ and which has the following properties. The coefficients of the Taylor expansion of 
$Q_{\Pi_p}^\ell(\Phi)$ in powers of the components of $\Phi$ and  
its derivatives satisfy that for any $j \geq 0$, there exists a constant $C_j > 0$, independent of $ \Pi_p \in G_{k+1}(TM) $ and 
$\e \in (0,1)$, 
\begin{multline*}
\| Q_{\Pi_p}^\ell(\Phi_1) - Q_{\Pi_p}^\ell(\Phi_2) \|_{\mathcal C^{j,\alpha}(S^k)} \leq \\
c \, \left( \| \Phi_1 \|_{\mathcal C^{j+2,\alpha}(NS^k)} + \| \Phi_2 \|_{\mathcal C^{j+k,\alpha}(NS^k) } \right)^{\ell-1} \| \Phi_1 - \Phi_2 \|_{\mathcal C^{j+k,\alpha}(NS^k)}
\end{multline*}
provided $\| \Phi_i \|_{\mathcal C^1(NS^k)} \leq 1$, $i = 1,2 $.
\end{notation}

\medskip

Using that the Christoffel symbols of the metric $ g_\e $ are of order $\mathcal O(\e^2)$, we obtain
\begin{equation*}
\begin{split}
g_\e \left( H^{g_\e}(S^k_\Phi), n_\Phi^\e \right) - k & = - \mathcal Ric_{k+1}(\Pi_p)(\Theta,\Theta) + J_{S^k}^\parallel \phi \\
 & + \mathcal O(\e^3) + \e^2 \, L_{\Pi_p}(\Phi) + Q_{\Pi_p}^2(\Phi), \\
g_\e \left( H^{g_\e}(S^k_\Phi), (\mathcal N_\Phi)_\mu^\e \right) & = - \mathcal Ric_{k+1}^\perp(\Pi_p)(\Theta,E_\mu) + \left( J_{S^k}^\bot - D_{B^{k+1}} \right) \, \Phi^\mu  \\ & + \mathcal O(\e^3) + \e^2 \, L_{\Pi_p}(\Phi) + Q_{\Pi_p}^2(\Phi).
\end{split}
\end{equation*}

\medskip
 
As before, we let $ \mathcal K^{\parallel} $ and $ \mathcal K^\bot $ be the null-spaces of the operators 
$$ 
J_{S^k}^\parallel = \Delta_{S^k} + k \quad \mbox{and} \quad L_{B^{k+1}}^\bot = \Delta_{S^k} - D_{B^{k+1}}
 $$
and write $ \mathcal P^\parallel $ and $\mathcal P^\bot$ for the orthogonal complements of $\mathcal K^\parallel$ and 
$\mathcal K^\bot$ in $L^2(S^k)$.  Define
\begin{equation}
\mathfrak E_{\e,\Pi_p}:= T_pM \times (T_pM \oplus \mathbb R)^{m-k} \times \mathcal P^\parallel \times (\mathcal P^{\bot})^{m-k}
\label{frakE}
\end{equation}
There exists an operator 
$$ 
\mathcal G_{\e,\Pi_p}: (\mathcal C^{0,\alpha}(S^k) )^{m-k} \longrightarrow \mathfrak E_{\e,\Pi_p}
$$
such that 
\begin{multline*}
\mathcal G_{\e,\Pi_p} (f_0,f_1,\ldots, f_{m-k}) \\= \left( \vec a(\e,\Pi_p,f), \vec c_\mu(\e,\Pi_p,f), d_\mu(\e,\Pi_p,f), \phi(\e,\Pi_p,f), \Phi^\bot(\e,\Pi_p,f) \right)
\end{multline*}
is the solution to
$$ \left\{ \begin{array}{l} J_{S^k}^\parallel \, \phi = g_p(\vec a, \Theta) + f_0 \\[3mm]
 L_{B^{k+1}}^\bot \, \Phi^\mu = g_p(\vec c_\mu, \Theta) + d_\mu + f_{\mu-k}.
 \end{array} \right. $$
Applying a standard fixed point theorem for contraction mappings, we find that there exist constants $ c \in \mathbb R$ and 
$ \e_0 \in (0,1) $ such that for every $\e \in (0,\e_0) $ and $ \Pi_p \in G_{k+1}(TM) $ there is a unique
$$ 
\left( \vec a(\e,\Pi_p), \vec c_\mu(\e,\Pi_p), d_\mu(\e,\Pi_p), \phi(\e,\Pi_p), \Phi^\bot_{\e,\Pi_p} \right) \in \mathfrak E_{\e,\Pi_p}.
$$
(the indices are suppressed for simplicity) which belongs to a closed ball of radius $ c \, \e^2 $ in $\mathfrak E_{\e,\Pi_p}$ and such that 
$$
H^{g_\e}(S^k_{\Phi}) = - k \, n^\e_{\Phi} + g_p(\vec a, \Theta) \, n^\e_{\Phi} + \sum_{\mu=k+1}^{m+1} \, \left( g_p(\vec c_\mu, \Theta) + 
d_\mu \right) \, ( \mathcal N_{\Phi} )^\e_\mu.
$$
Putting
$$ n_K = (F_\e)_* \, n^\e_\Phi \quad \mbox{and} \quad N_\mu = (F_\e)_* \, ( \mathcal N_\mu )_\Phi^\e $$
and taking $K_\e(\Pi_p) := F_\e( S^k_{\Phi(\e,\Pi_p)} )$ and $ Q_\e(\Pi_p):= F_\e(B^{k+1}_{\e,\Phi(\e,\Pi_p)})  $ finishes the proof.
\begin{remark}
 Notice that 
 \begin{center} \scalebox{0.9}{ $ \mathcal Ric_{k+1}(\Pi_p)(\Theta,\Theta) \in \mathcal P^\parallel \quad \mbox{and} \quad \mathcal Ric_{k+1}^\bot(\Pi_p)(\Theta,E_\mu) \in \mathcal K^\bot $}
 \end{center}
Moreover, it was remarked in \cite{Pac-Xu} that
\[
\begin{array}{ll} \mathcal Ric_{k+1}(\Pi_p)(\Theta,\Theta) & = \sum \limits_{a=1}^{k+1} \mathcal Ric_{k+1}(\Pi_p)_{aa} \, (\Theta^i)^2 + 
  \sum \limits_{a \neq b=1}^{k+1} \mathcal Ric_{k+1}(\Pi_p)_{ab} \, \Theta^a \, \Theta^b \\[5mm] 
 & = \frac{1}{k+1} \mathcal R_{k+1}(\Pi_p) + \breve{{\mathcal Ric}}_{k+1}(\Pi_p)(\Theta,\Theta) \end{array}
\]
where $ \breve{{\mathcal Ric}}_{k+1}(\Pi_p)(\Theta,\Theta)$ belongs to the eigenspace of $\Delta_{S^k}$ associated to the eigenvalue 
$2(k+1)$. Using this, one can easily verify that
\[
\phi_{\e,\Pi_p}(\Theta) = - \frac{\e^2}{3} \, \left( \frac{2}{k(k+2)} \, \mathcal R_{k+1}(\Pi_p) - \frac{1}{k+2} \, \mathcal Ric_{k+1}(\Pi_p)(\Theta,\Theta) \right) \, \Theta + \mathcal O(\e^3), 
\]
\[
[\Phi]^\perp_{\e,\Pi_p} = \mathcal O(\e^3).
\]
\end{remark}
\end{proof}

\subsection{The variational argument}
We now employ a variational argument to prove that one can choose $\Pi_p \in G_k(M)$ in such a way
that the submanifold $K_\e(\Pi_p)$ obtained in the previous Proposition has constant mean curvature. 

\medskip

To state our result, we introduce the following restrictions of the Riemann tensor of $M$:

\begin{equation*}
\begin{split}
& R_{k+1}(\Pi_p)(v_1,v_2,v_3,v_4) = g_p(R_p(v_1,v_2)v_3,v_4), \qquad v_1, v_2, v_3, v_4 \in \Pi_p, \\
& R_{k+1}^\bot(\Pi_p)(v_1,v_2,v_3,N) = g_p(R_p(v_1,v_2)v_3,N), \qquad v_1, v_2, v_3 \in \Pi_p,  \quad N \in \Pi_p^\bot, 
\end{split}
\end{equation*}

\medskip

Finally, introduce the function $\mathbf r$ on $G_{k+1}(TM)$:
\[
\begin{array}{ll} 
\mathbf r( \Pi_p ) & = \frac{1}{ 36(k+5) } \Big( 8 \, \| \mathcal Ric_{k+1}(\Pi_p) \|^2 - 18 \, \Delta_{k+1}^g \mathcal R_{k+1}(\Pi_p) 
- 3\, \| R_{k+1}(\Pi_p) \|^2  \\[3mm] & + 5 \, \mathcal R^2_{k+1}(\Pi_p) + 8 \, \| \mathcal Ric_{k+1}^\bot(\Pi_p) \|^2 + 12 \, \| 
R^\bot_{k+1}(\Pi_p) \|^2 \Big) \\[3mm] 
& +\frac{1}{9(k+2)} \Big( \frac{k+6}{k} \, \mathcal R_{k+1}^2(\Pi_p) - 2 \, \| \mathcal Ric_{k+1}(\Pi_p) \|^2 \Big)  \\[3mm]
& - \frac {4 k}{3(k+3)(k+5)} \, \| \mathcal Ric_{k+1}^\bot(\Pi_p) \|^2 
\end{array}  
\]
where $ \Delta_{k+1}^g T(\Pi_p) = \sum_{i=1}^{k+1} \, \nabla_{E_i}^2 T(p) $, for any tensor $T$ on $M$.

\medskip

Now consider the energy $\mathcal E_{\varepsilon}$ restricted to this finite dimensional space of submanifolds,
\[
\mathcal E_\varepsilon (\Pi_p) := {\rm Vol}_k (K_\varepsilon (\Pi_p)) - \frac{k}{\varepsilon} \, {\rm Vol}_{k+1} (Q_\varepsilon (\Pi_p)),
\]
which is a function on $G_{k+1}(TM)$.  Tracing through the construction of $K_\varepsilon (\Pi_p)$ one obtains the
relationship of this function to the curvature functions defined above.
\begin{lemma} \label{expEner}
There is an expansion
\medskip
\[
\frac{(k+1) \, \mathcal E_\varepsilon (\Pi_p) }{ \e^k \, \mathrm{Vol}(S^k)} = \left( 1 -  \frac{ \e^2 }{2(k+3)} \, 
\mathcal R_{k+1} (\Pi_p) + \frac{\e^4}{2(k+3)} \, \mathbf r(\Pi_p) + \mathcal O(\e^5) \right)
\]
\end{lemma}
\begin{proof}
The proof is a technical calculation, contained in the Appendix.
\end{proof}

\medskip

The main result of this section is the following proposition 
\begin{proposition} \label{CritPoint}
If $\Pi_p$ is a critical point of $\mathcal E_\e$, then $ K_\e(\Pi_p) $ has constant mean curvature.
\end{proposition}
\begin{remark}
Theorems \eqref{T1} and \eqref{T2} are Corollaries of Proposition \eqref{CritPoint}. Indeed, if we define
\begin{equation}
\Psi(\e,\Pi_p) = 2 \, \e^{-2} \, (k+3) \left( 1 - (k+1)\dfrac{\mathcal E_\e(\Pi_p) }{\e^k\mathrm{Vol}(S^k)} \, \right).
\label{defPsi}
\end{equation}
then for any $j \geq 0$, there exists a constant $C_j$ which is independent of $\e$ such that 
$$ 
\| \Psi(\cdot, \e) - \mathcal R_{k+1} + \mathcal \e^2 \mathbf r(\Pi_p) \|_{\mathcal C^j(G_{k+1}(TM))} \leq C_j \, \e^3; 
$$
\end{remark}

\begin{proof}[Proof of the Proposition]
Let $\Pi_p$ be a critical point of $\mathcal E_\e$.  We show that the parameters $\vec a$, $\vec c$ and $d$ must
then necessarily vanish.  We do this by considering the various types of perturbations of $\Pi_p$.

First consider the perturbations in $G_{k+1}(M)$ which correspond to 
parallel translations of $\Pi_p$.  In other words, we suppose that the family of planes $\Pi_{\mathrm {\exp}_p( t \xi )}$ in $G_{k+1}(M)$ 
are parallel translates of $\Pi_p$ along the geodesic $ \exp_p(t \xi)$.
  
\medskip
  
The submanifold $K_\varepsilon(\Pi_{\mathrm {exp}_p( t\vec \xi )})$ is a normal graph over $K_\e(\Pi)$ by a vector field $\Psi_{\e, \Pi_p, \xi, t} $
which depends smoothly on $t$. This defines a vector field on $K_\varepsilon (\Pi_p)$ by 
$$ 
Z_{\e, \Pi_p,  \xi} = \left. \partial_t \Psi_{\e,\Pi_p, \xi,t}\right|_{t=0}. 
$$
  
\medskip
  
The first variation of the volume formula yields
\begin{equation} \label{d} \begin{array}{rl}
0 & = D\mathcal E_\varepsilon|_{\Pi_p}(\xi) \\[3mm]
& = \int_{K_\varepsilon(\Pi_p)} \Big( g( H(K_\e(\Pi_p)),Z_{\e, \Pi_p, \xi} ) - \frac{k}{\e} \, g(n, Z_{\e, \Pi_p, \xi}) \Big) \, dvol_{K_\e(\Pi_p)} \\[3mm] 
& \qquad - \displaystyle \frac{k}{\e} \int_{Q_\e(\Pi_p)} g( H(Q_\e (\Pi_p)), Z_{\e, \Pi_p, \xi} ) \ dvol_{Q_\e(\Pi_p)},
\end{array}
\end{equation}
and then the construction of $Q_\varepsilon (\Pi_p)$ and $K_\varepsilon (\Pi_p)$ gives that
\begin{multline*}
\int\Big( g_p( \vec a, \Theta) g(n,Z_{\e,\Pi_p, \xi}) \\
+ \sum \limits_{\mu=k+1}^{m+1} \left( g_p( \vec c_\mu, \Theta) + d_\mu \right) g(Z_{\e, \Pi_p, \xi},N_\mu) \Big) \ dvol_{K_\e (\Pi_p)} = 0.
\end{multline*}

Let $ \Xi$ be the vector field obtained by parallel transport of $\xi$ along geodesics issuing from $p$, and suppose that $c$
is a constant independent of $\varepsilon$ and $ \xi$. Then 
$$ 
\| Z_{\e, \Pi_p, \xi} - \Xi \|_g \leq c \, \e^2 \, \| \xi \|.
$$

\medskip

By construction of $K_\e(\Pi_p)$, we have 
$$ 
\| n + (F_\e)_* \Theta \|_g \leq c \, \varepsilon^2, \quad \mbox{and} \quad \| N_\mu - (F)_* E_\mu \|_g \leq c \, \varepsilon^2.
$$
    
\medskip

Now take $ \xi \in \Pi_p \subset TM_p $, so that 
$$ 
g(n,Z_{\e,\Pi_p, \xi}) = g \left( -(F_\e)_* \Theta + \left( n + (F)_* \Theta \right), \ \Xi + \left( Z_{\varepsilon,\Pi_p, \xi} - \Xi  \right) \right), 
$$
and
$$ 
g( N_\mu,Z_{\e, \Pi_p, \xi} ) = g \left( (F_\e)_*E_\mu + \left( N_\mu - (F_\e)_*\vec E_\mu, \ \Xi + \left( Z_{\e,\Pi_p, \xi} - \Xi \right) \right) \right). 
$$

\medskip

Using the expansion of $g$ near $p$, we conclude that
$$ 
\left| g(n,Z_{\e,\Pi_p, \xi}) + g_p( \xi,\Theta) \right| \leq c \, \e^2 \| \xi \|, \quad \mbox{and} \quad \left| g( N_\mu, Z_{\e, \Pi_p, \xi} ) \right| 
\leq c \, \e^2 \| \xi \|,
$$
hence 
\[
\begin{array}{rl} 
\displaystyle \int_{K_\varepsilon (\Pi_p)} & g_p( \vec a, \Theta ) g_p( \xi, \Theta)  \\  & 
\leq \Big| \displaystyle \int_{K_\e (\Pi_p)} g_p( \vec a, \Theta) g_p( \xi, \Theta) + \displaystyle \int_{K_\varepsilon (\Pi_p)} g_{p}( \vec a,\Theta) \, 
g(Z_{\varepsilon,\Pi_p,\xi}, n ) \\[5mm] 
& \qquad + \sum \limits_{\mu=k+1}^{m+1} \displaystyle \int_{K_\e (\Pi_p)} \left( g_p( \vec c_\mu, \Theta ) + d_\mu \right) g(Z_{\varepsilon, \Pi_p, \xi}, N_\mu) \Big| \\[5mm]
& \left. \leq c \, \e^2 \, \| \xi \| \Big( \displaystyle \int_{K_\e (\Pi_p)} \left| g_p (\vec a,\Theta) \right| \right. + \sum \limits_{\mu=k+1}^{m+1} \displaystyle \int_{K_\e (\Pi_p)} \left| g_p(\vec c_\mu,\Theta) + d_\mu \right| \Big) 
\end{array}
\]
Now let $ \xi = \vec a $, so that 
\begin{multline*}
\int_{K_\e (\Pi_p)} \left| g_p( \vec a,\Theta) \right|^2  \\
\leq c \, \e^2 \| \vec a \|  \left( \int_{K_\e (\Pi_p)} \left| g_p(\vec a, \Theta) \right| + \sum \limits_{\mu = k+1}^{m+1} \int_{K_\e (\Pi_p)} | g_{p}( \vec c_\mu, \Theta ) + d_\mu | \right) \label{ineg1}
\end{multline*}

In Euclidean space there is an equality
$$ 
\mathrm {Vol}_{k}(S^k)\| v \|^2 = (k+1) \int_{S^k} \langle v, \Theta \rangle^2,  \quad \mbox{for all}\ \ v \in \mathbb R^k.
$$
By the expansion of the induced metric, we obtain for $\e$ small enough
$$ 
\frac{1}{2} \ \mathrm {Vol}_{k}(S^k) \, \e^k \, \|  v \|^2 \leq (k+1) \int_{K_\varepsilon (\Pi_p)} | g_p( v, \Theta) |^2. 
$$
Also, because $ \mathrm {Vol}_{k}(K_\varepsilon (\Pi_p)) = \mathcal O(\varepsilon^k) $, we deduce
\begin{equation} \label{en1}
 \| \vec a \| \leq c \, \varepsilon^2 \Big( \| \vec a \| +  \sum_{\mu=k+1}^{m+1} \left( \| \vec c_\mu \| + | d_\mu | \right) \Big).
\end{equation}

\medskip
 
Now move $p$ in the direction of a vector $ \xi \in \Pi_p^\bot $ to get
$$ 
\left| g(Z_{\e,\Pi_p,\xi},N_\mu) - g_p(\xi, E_\mu) \right| \leq c \, \e^2 \| \xi \|, \quad \mbox{and} \quad | g(n,Z_{\e,\Pi_p,\xi}) | 
\leq c \, \e^2 \| \xi \|.
$$

\medskip
 
Thus we can write
\[
\begin{array}{ll} \sum \limits_{\mu=1}^{m-k} \displaystyle \int_{ K_\e (\Pi_p) } & \left( g_p ( \vec c_\mu,\Theta) + d_\mu \right) \, g_p( \xi, E_{\mu} ) \\[5mm] 
& \leq \Big| \sum \limits_{\mu=k+1}^{m+1} \displaystyle \int_{K_\e (\Pi_p) } \left( g_p ( \vec c_\mu, \Theta ) + d_\mu \right) g(Z_{\e, \Pi_p, \xi}, N_\mu ) \\[5mm] 
& - \sum \limits_{\mu=k+1}^{m+1} \displaystyle \int_{K_\e (\Pi_p)} \left( g_p ( \vec c_\mu, \Theta ) + d_\mu \right) g_p( \xi, E_{\mu}) \\[5mm]
& + \displaystyle \int_{ K_\e (\Pi_p) } g_p( \vec a, \Theta ) g(Z_{\e, \Pi_p, \xi}, n ) \Big| \\[5mm]
& \leq  c \, \e^2 \| \xi \| \displaystyle \int_{K_\e (\Pi_p)} 
\Big(|g_p( \vec a, \Theta )| + \sum \limits_{\mu=k+1}^{m+1} | g_p ( \vec c_\mu, \Theta)  + d_\mu | \Big) \end{array}. 
\]
Taking $ \xi = d_\nu \, E_\nu $ gives
\begin{equation} \label{en2} \begin{array}{ll}
 \ \displaystyle \int_{K_\e (\Pi_p)} d_\nu \ g_p (\vec c_\nu, \Theta) + {d_\nu}^2 & \leq c \, \e^2 |d_\nu| \Big( \int_{K_\e (\Pi_p)}| g_p( \vec a, 
\Theta) | \\[3mm] 
 & + \sum \limits_{\mu=k+1}^{m+1} \displaystyle \int_{K_\e (\Pi_p)} | g_p( \vec c_\mu, \Theta ) + d_\mu | \Big) \end{array}
\end{equation}

\medskip

Next consider a perturbation of $\Pi_p$ by a one-parameter family of rotations of $\Pi_p$ in $T_p M$ generated 
by an $(m+1) \times (m+1)$ skew matrix $A$. Then 
$$ 
\left. D \mathcal E_\varepsilon \right|_{\Pi_p} (A) = \left.\frac{d}{dt} \right|_{t=0} \mathcal E_\varepsilon( (I + tA + O(t^2))\Pi_p ) = \left.\frac{d}{dt} \right|_{t=0} \mathcal E( A_t(K_\varepsilon (\Pi_p)) ), 
$$
where, in geodesic normal coordinates 
\[
A_t(x) = x + tA x + \mathcal O(t^2).
\]
The coordinates of the vector field associated to this flow are
$$ 
Z_{\e, \Pi_p, \xi }(x) = \left. \frac{d}{dt} \right|_{t=0} A_t (x) = Ax. 
$$
Considering only matrixes $ A \in \mathfrak o(m)$ such that $A: \Pi_p \to \Pi_p^\bot $, we obtain
$$ 
\left| g(Z_{\e,\Pi_p,\xi},n) \right| \leq c \, \e^2 \| A \Theta \|, \quad \mbox{and} \quad \left| g(Z_{\e, \Pi_p, \xi}, N_\mu) - g_p(A \Theta, E_\mu) \right| \leq c \, \e^2 \| A \Theta \|. 
$$
This gives, then, 
\[
\begin{array}{ll} \sum \limits_{\mu=k+1}^{m+1} & \displaystyle \int_{K_\e (\Pi_p)} \left( g_p(\vec c_\mu, \Theta) + d_\mu \right) g_p (A \Theta, E_\mu) \\[5mm]
& \leq  \Big| \sum \limits_{\mu=k+1}^{m+1} \displaystyle \int_{K_\e (\Pi_p)} \left( g_p( \vec c_\mu, \Theta) + d_\mu \right) g(Z_{\e, \Pi_p, \xi}, N_\mu) \\[5mm]
& - \sum \limits_{\mu=k+1}^{m+1} \displaystyle \int_{K_\e (\Pi_p)} \left( g_p(\vec c_\mu, \Theta) + d_\mu \right) g_p(A \Theta, E_\mu) \\[5mm]
& + \displaystyle \int_{K_\e (\Pi_p)} g_p( \vec a, \Theta) g(Z_{\e, \Pi_p, \xi}, n) \Big| \\[5mm]
& \leq c \, \e^2 \displaystyle \int_{K_\e (\Pi_p)} \left(\| A \Theta \| \ | g_p(\vec a,\Theta) | + \sum_{\mu=k+1}^{m+1} \| A \Theta \| \ | g_p ( \vec c_\mu, \Theta) + d_\mu |\right). 
\end{array}
\]

Let $C_\nu$ be the $(m-k)\times(k+1)$ matrix with column $\nu$ equal to the vector $ \vec c_\nu \in \mathbb R^{k+1}$,
and all other columns equal to $0$. Then if
$$ 
A = \left( \begin{array}{cc} 0 & -C_\nu^T \\ C_\nu & 0 \end{array} \right), 
$$
we get
\begin{equation} \label{en3} 
\begin{array}{rl} 
\displaystyle \int_{K_\e (\Pi_p)} g_p(\vec c_\nu, \Theta)^2 & + \ g_p(\vec c_\nu,\Theta) d_\nu \leq C \e^2 
\Big( \int_{K_\varepsilon (\Pi_p)} | g_p(\vec c_\nu,\Theta) | \ | g_p(\vec a,\Theta) | \\[5mm]
& + \sum \limits_{\mu=k+1}^{m+1} \displaystyle \int_{K_\e (\Pi_p)} | g_p(\vec c_\nu,\Theta) |\left| g_p(\vec c_\mu,\Theta) + d_\mu \right| \Big)
\end{array}
\end{equation}

Adding \eqref{en2} and \eqref{en3} now gives
\begin{multline*}
 \int_{K_\e (\Pi_p)}| d_\nu + g_p(\vec c_\nu,\Theta) |^2 \leq c \, \e^2 \Big( \int_{K_\e (\Pi_p)} \left( |d_\nu| + | g_p(
 \vec c_\nu,\Theta) | \right)| g_p( \vec a, \Theta) | \\[3mm]
 + \sum_{\mu=k+1}^{m+1} \left( |d_\nu| + | g_p(\vec c_\nu,\Theta) \right| ) \left| g_p(\vec c_\mu,\Theta) + d_\mu \right| \Big)
\end{multline*}

In Euclidean space, if $v \in \RR^{k+1}$ and $\alpha \in \RR$ are arbitrary, then 
$$ 
\int_{S^k} \left| \alpha + \langle v, \Theta \rangle \right|^2  = \left( \alpha^2 + \frac{1}{k+1}\| v \|^2 \right)\mathrm {Vol}_{k}(S^k).
 $$
Using, once again, the decomposition of the induced metric on $K_\varepsilon (\Pi_p)$, we see that when $\e$ is small enough,
\begin{equation} 
\frac{1}{2(k+1)} \varepsilon^{k} \, \mathrm {Vol}_{k} (S^k) \left( \alpha^2 + \| v \|^2 \right) \leq \int_{K_\varepsilon (\Pi_p)} \left| \alpha + g_p( v, \Theta) \right|^2. 
\end{equation}
which give
\begin{multline*}
\| \vec c_\nu \|^2 + |d_\nu|^2  \\ 
\leq c \, \dfrac{1}{\e^{k-2}} \left( \| \vec c_\nu \| + | d_\nu | \right) \Big( \displaystyle \int_{K_\e (\Pi_p)}| g_p( \vec a, \Theta) | + \sum \limits_{\mu=1}^{m-k} \int_{K_\e (\Pi_p)} | g_p( \vec c_\mu, \Theta) + d_\mu |  \Big)
\end{multline*}
Since $ \mathrm{Vol}_{k} (K_\varepsilon (\Pi_p)) = \mathcal O(\varepsilon^k) $, we get
\begin{equation}
 \ \| \vec c_\nu \| + |d_\nu| \leq c \, \e^2 \, ( \| \vec a \| + \sum_{\mu=1}^{m-k} \left( \| \vec c_\mu \| + | d_\mu | \right) ) \label{en4}
 \end{equation}
Adding \eqref{en1} and \eqref{en4} gives 
$$ 
\left( \| \vec a \| + \sum_{\mu=1}^{m-k} \left( \| \vec c_\mu \| + | d_\mu | \right) \right) \leq c \, \e^2 \left( \| \vec a \| + \sum_{\mu=1}^{m-k} ( \| \vec c_\mu \| + | d_\mu | ) \right), 
$$
which implies finally that $ \|  \vec a \| = 0 $, $ \| \vec c_\mu \| = 0 $ and $ |d_\mu| = 0$, $k+1 \leq \mu$. 

\medskip

We conclude that if $\Pi_p$ is a critical point of the functional $\mathcal E_\varepsilon$, then the manifold 
$ K_\varepsilon (\Pi_p)  $ is a constant mean curvature submanifold of $M$.
\end{proof}

\section{Appendix 1}

\noindent {\bf Mean curvature of submanifolds:}  
Let $ \Sigma^k \subset M^{m+1}$ be an embedded submanifold. Let $x^1, \ldots, x^k$ be local coordinates on $\Sigma$ and 
\[
E_\alpha  = \partial_{x_\alpha}
\]
the corresponding coordinate vector fields.  Suppose that $E_{k+1}, \ldots, E_{m+1}$ is a local frame for
$N\Sigma$. This gives local coordinates transverse to $\Sigma$ by 
\[
p \in \Sigma \longmapsto {\exp}_p (\sum_{j=k+1}^{m+1} x^j \, E_j)
\]
We make the convention that Greek indices run from $1$ to$ k$, while Latin indices run from $k+1$ to $m+1$. 
The induced metric on $\Sigma$ has coefficients $\bar g_{\alpha \beta}$, while 
\[
\bar h^i_{\alpha \beta}  : = \Gamma^i_{\alpha \beta}  = g ( \nabla_{E_\alpha} E_\beta , E_i) 
\]
are the coefficients of the shape operator. We also record the Christoffel symbols 
\[
\Gamma^j_{\alpha i} = g ( \nabla_{E_\alpha} E_i , E_j)
\]

The following result is standard, cf.\ \cite{Mah-Maz-Pac} for a proof.
\begin{lemma}
If $ X = \sum \limits_{j=k+1}^{m+1} x^j \, E_j $, then 
\begin{center}
\scalebox{0.9}{$
\begin{array}{rl}
g_{\alpha \beta} & = \bar g_{\alpha \beta} - 2 \, \bar g ( \bar h_{\alpha \beta}, X ) + g( R(E_\alpha,X)E_\beta,X ) + g( \nabla_{E_\alpha} X, \nabla_{E_\beta} X ) + \mathcal O(|x|^3) \\[3mm]
& = \bar g_{\alpha \beta}  -  2 \, \bar h_{\alpha \beta}^i \, x^i + \left(  g( R (E_\alpha, E_i) E_\beta , E_j ) + g^{\gamma \gamma'}\bar h^i_{\alpha \gamma}Ê\, \bar h^j_{\gamma' \beta} + \Gamma^i_{\alpha \ell} \, \Gamma^j_{\ell  \beta}  \right) \, x^i \, x^j + \mathcal O (|x|^3) \\[3mm]
g_{\alpha j} & = - \Gamma_{\alpha j}^i \, x^i + \mathcal O (|x|^2) \\[3mm]
g_{ij} & =  \delta_{ij}Ê+ \frac{1}{3}Ê\, g(R(E_i, E_\ell) E_j, E_{\ell'}) \, x^\ell \, x^{\ell'} + \mathcal O (|x|^3)
 \end{array}$}
\end{center}
\end{lemma}

Let $\Phi$ be a smooth section of $N\Sigma$ and consider the normal graph $\Sigma_\Phi = \{{\exp}_p ( \Phi(p)): p \in \Sigma \}$.
Now let us use the previous lemma to expand the metric and volume form on $\Sigma_\Phi$.  To state this result properly, 
introduce $\nabla^N$, the induced connection on $N\Sigma$,
\[
\nabla^N \Phi = \pi_{N\Sigma} \circ \nabla \Phi
\]
Using the definitions of \S 2, we find that
\begin{lemma}
\[
\begin{array}{rllll}
{\rm Vol_k} (\Sigma_\Phi) & = & \displaystyle {\rm Vol}_k (\Sigma) - \int_\Sigma g(H (\Sigma), \Phi) \, {\rm dvol}_\Sigma \\[3mm]
& + &  \displaystyle \frac{1}{2}Ê\int_\Sigma \left( |\nabla^N \Phi|_g^2 -  g( (\Ric_\Sigma + \mathfrak H_\Sigma^2 ) \, \Phi , \Phi)  \right)\, {\rm dvol}_\Sigma \\[3mm]
& + &  \displaystyle \frac{1}{2}Ê\int_\Sigma ( g(H(\Sigma), \Phi))^2 \, {\rm dvol}_\Sigma + ...
\end{array}
\]
\end{lemma}
\begin{proof}
First of all we expand the induced metric on $\Sigma_\Phi$. Using the result of the previous Lemma, we find
\begin{align*}
 (\bar g_\Phi)_{\alpha \beta} & = \bar g_{\alpha \beta} - 2  \, g (\bar h_{\alpha \beta} , \Phi) + g( R (E_\alpha  , \Phi) \,E_\beta  , \Phi) + g( 
 \nabla_{E_\alpha} \Phi, \nabla_{E_\beta} \Phi ) + \ldots \\[3mm]
 & = \bar g_{\alpha \beta} - 2  \, g (\bar h_{\alpha \beta} , \Phi) + g( R (E_\alpha  , \Phi) \,E_\beta  , \Phi) \\[3mm] 
 & + \bar g^{\gamma \gamma'}  g(\bar h_{\alpha \gamma} , \Phi) \, g( \bar h_{\gamma \beta} , \Phi) + g( \nabla^N_{E_\alpha} \Phi, \nabla_{E_\beta}^N \Phi) + \ldots
\end{align*}
Now use the well known expansions
\[
{\rm det} (I+ A) = 1 + {\rm Tr} \, A + \frac{1}{2} \, \left( ({\rm Tr} A)^2- {\rm Tr} (A^2) \right) + ...
\]
together with $\sqrt{1+x} = 1 + \frac{1}{2} \, x - \frac{1}{8} \, x^2 + ...$ to conclude that
\[
\begin{array}{rlllll}
\sqrt{{\rm det} \, \bar g_\Phi}= \left(  1 -  g(H (\Sigma), \Phi) +  \frac{1}{2}Ê \left( |\nabla^N \Phi|_g^2  \right. \right.  -  g( (\Ric_\Sigma + (\mathfrak H)_\Sigma^2 ) \, \Phi , \Phi) \\[3mm]
  + \displaystyle \left. \left. ( g(H(\Sigma), \Phi))^2 \right) + ...  \right) \, \sqrt{{\rm det} \, \bar g}
\end{array}
\]
This completes the proof.
\end{proof}

From this we obtain the first and second variations of the volume functional,
\begin{equation}
D_\Phi {\rm Vol_k} (\Sigma_\Phi)|_{\Phi} \Psi  =  - \int_\Sigma g(H (\Sigma_\Phi), \Psi) \, {\rm dvol}_{\Sigma_\Phi},
\label{eq:fd}
\end{equation}
and 
\[
\begin{array}{rllll}
D^2_\Phi {\rm Vol_k} (\Sigma_\Phi)|_{\Phi =0} (\Psi, \Psi) &  = & \displaystyle  \int_\Sigma \left( |\nabla^N \Psi|^2 -  g((\Ric_\Sigma  + \mathfrak H_\Sigma^2 ) \, \Psi , \Psi)  \right)\, {\rm dvol}_\Sigma  \\[3mm]
& + &   \displaystyle \int_\Sigma ( g(H(\Sigma), \Psi))^2 \, {\rm dvol}_\Sigma.
\end{array}
\]

On the other hand, differentiating (\ref{eq:fd}) once more gives
\[
\begin{array}{rllll}
D^2_\Phi {\rm Vol_k} (\Sigma_\Phi)|_{\Phi =0} (\Psi , \Psi)  & = &  - \displaystyle \int_\Sigma g(D_\Phi H (\Sigma_\Phi)|_{\Phi =0} \Psi, \Psi) \, {\rm dvol}_\Sigma \\[3mm]
& + & \displaystyle  \int_\Sigma (g(H (\Sigma), \Psi) )^2 \, {\rm dvol}_K.
\end{array}
\]
Comparing the two formul\ae\ implies that the orthogonal projection of the Jacobi operator to $ N \Sigma$ equals
\[
J_\Sigma^N := D_\Phi H (\Sigma_\Phi)|_{\Phi =0}  =  \Delta^N_g +  \Ric^N_\Sigma + \mathfrak H_\Sigma^2 \, ,
\]

\section{Appendix 2}
Let $ K_\e(\Pi_p) $ be the constant mean curvature submanifold constructed in Proposition \eqref{expVol} and denote by
$F : T_pM \longrightarrow M$ the exponential mapping. Recall that
$$ K_\e(\Pi_p) = F(S^k_{\e,\Phi}), $$
where $ S^k_{\e,\Phi} $ is a submanifold of $ T_p M $ parametrized by 
$ \left\{ \e \, (1-\phi) \, \Theta + \e \, \Phi^\bot, \ \Theta \in S^k \right\} $. It follows from the proof of that 
proposition that 
\begin{align*} 
 & \phi(\Theta) = \dfrac{\e^2}{3} \left( \dfrac{2}{k(k+2)} \mathcal R_{k+1}(\Pi_p) - \dfrac{1}{k+2} \mathcal Ric(\Pi_p)(\Theta,\Theta) \right) + \mathcal O(\e^3), \\[3mm] 
 & \Phi^\bot = \mathcal O(\e^3). 
\end{align*}

\medskip

There is also the minimal submanifold $ Q_\e(\Pi_p) = F( B^{k+1}_{\e,\Phi} ) $,
\medskip
\newline where $  B^{k+1}_{\e, \Phi} = \left\{  \e \, y + \e \, U_\Phi(y), \ y \in B^{k+1} \right\} $ and
\begin{multline*} U_\Phi(y) = \phi \left( y / \| y \| \right) + W(y) + \mathcal O(p)(\e^3), \\[3mm] 
W(y) = \dfrac{1}{(k+3)} \, \sum_{i=1}^{k+1} \mathcal Ric^\perp(\Pi_p)_{i \mu} (|y|^2 - 1) \, y_i \, E_\mu. \end{multline*}

We shall calculate the volume forms of $ S^k_{\e,\Phi} $ and $B^{k+1}_{\e,\Phi} $ with respect to $ F^*g $. To prepare for this,
recall that near $x=0$
\begin{align*} \label{metric}
 (F^*g)_{ij} & = \delta_{ij} + \dfrac{1}{3} \, g_p(R_p(x,E_i)x,E_j) + \dfrac{1}{6} \, g_p(\nabla_x R_p(x,E_i)x,E_j) \\[3mm]
 & + \dfrac{1}{20} \, g_p(\nabla_x \nabla_x R_p(x,E_i)x,E_j) \\[3mm]
 & + \sum \limits_{\ell = 1}^{m+1} \dfrac{2}{45} \, g_p(R_p(x,E_i)x,E_\ell) \, g_p(R_p(x,E_j)x,E_\ell) + \mathcal O_p(|x|^5)
\end{align*}
where $ R_p $ is the curvature tensor of $ M $ at the point $p$, cf.\ \cite{Sch-Yau}.

\subsection{Volume of the CMC sphere} We first calculate the metric on $ S^k_{\e,\Phi} $. In terms of the coordinate
vector fields $ \Theta_\alpha, \alpha = 1,\ldots, k $ which are tangent to $S^k$, we can write the tangent vector fields to 
$S^k_{\e,\Phi}$ as
$$ 
\tau_\alpha = \e \, (1 - \phi(\Theta)) \, \Theta_\alpha - \e \, \partial_\alpha \, \phi \, \Theta + \sum_{\mu = k+1}^{m+1} \e \, \partial_\alpha \Phi^\mu \, E_\mu, \quad \alpha = 1, \ldots, k.
$$
The metric coefficients then equal 
\begin{center}
 \scalebox{0.9} {$ \begin{array}{ll} g^K_{\alpha,\beta} & = \e^2 (1 - \phi)^2 \, g^S_{\alpha, \beta} +  \e^2 \, \partial_\alpha \phi \, \partial_\beta \phi + \dfrac{\e^4}{3} \, (1 - \phi)^4 \, 
 g_p(R_p(\Theta,\Theta_\alpha)\Theta,\Theta_\beta) \\[3mm]
 & + \dfrac{\e^5}{6} \, g_p( \nabla_\Theta R_p(\Theta,\Theta_\alpha)\Theta, \Theta_\beta) + \dfrac{\e^6}{20} g_p( \nabla_\Theta \nabla_\Theta R_p(\Theta, \Theta_\alpha)\Theta, \Theta_\beta ) \\[5mm]
 & + \sum \limits_{l=1}^{k+1} \dfrac{2\e^6}{45} \, g_p( R_p(\Theta, \Theta_\alpha)\Theta, E_l) \, g_p( R_p(\Theta, \Theta_\beta)\Theta, E_l) \\[5mm]
 & + \sum \limits_{\mu = k+1}^{m+1} \dfrac{2\e^6}{45} \, g_p( R_p(\Theta, \Theta_\alpha)\Theta, E_\mu) \, g_p( R_p(\Theta, \Theta_\beta)\Theta, E_\mu) + \mathcal O(\e^7)     \end{array} $}
\end{center}

Using 
$$ \sqrt{\det(I+A)} = 1 + \dfrac{1}{2} \mathrm{tr}A + \dfrac{1}{8}(\mathrm{tr}A)^2 - \dfrac{1}{4} \mathrm{tr}(A^2) + \mathcal O(|A|^3), $$ 
we get
\begin{center}
 \scalebox{0.9}{ $ \begin{array}{ll}
 \e^{-k} \, \dfrac{\sqrt{\det g^K}}{\sqrt{ \det g^S }} & = 1 - k \phi + \dfrac{k(k-1)}{2} \phi^2 + \dfrac{1}{2} \, | \nabla_{S^k} \phi |^2 \\[5mm]
 & - \dfrac{\e^2}{6} (1-(k+2)\phi) \, \mathcal Ric_{k+1}(\Pi_p)(\Theta,\Theta) - \dfrac{\e^3}{12} \, \nabla_\Theta \, \mathcal Ric_{k+1}(\Pi_p)(\Theta,\Theta) \\[5mm]
 & - \dfrac{\e^4}{40} \nabla_\Theta^2 \, \mathcal Ric_{k+1}(\Pi_p)(\Theta,\Theta) + \dfrac{\e^4}{72} \left( \mathcal Ric_{k+1}(\Pi_p)(\Theta,\Theta) \right)^2  \\[5mm]
 & - \dfrac{\e^4}{180} \, \sum \limits_{i,j = 1}^{k+1} g_p(R_p(\Theta,E_i)\Theta, 
 E_j) ^2 \\[5mm] 
 & + \dfrac{\e^4}{45} \sum \limits_{i = 1}^{k+1} \, \sum \limits_{\mu = k+1}^{m+1} g_p(R_p(\Theta,E_i)\Theta, E_\mu )^2 + \mathcal O_p(\e^5).
\end{array} $ }
\end{center}

\subsection{Volume of the minimal ball} The tangent vectors to $B^{k+1}_{\e,\Phi}$ are
$$ T_i(y) = \e \, ( 1 - u(y) ) \, E_i + \e \, \partial_{y_i} \, u(y) \, y + \e \sum_{\mu = k+1}^{m+1} \partial_{y_i} \, W^\mu(y) \, E_\mu + 
\mathcal O_p(\e^4), $$
where $ u(y) = \phi(y/|y|) $. The corresponding metric coefficients are
\begin{center}
 \scalebox{0.9}{ $ \begin{array}{ll}
 \e^{-2} \, g^Q_{ij} & = ( 1 - u )^2 \, \delta_{ij} + (1-u) \left( \partial_{y_i} u \, y_j + \partial_{y_j} u \, y_i \right) + |y|^2 \partial_{y_i} u \, \partial_{y_j} u + \sum \limits_{\mu = k+1}^{m+1} \partial_{y_i} W^\mu \, \partial_{y_j} W^\mu \\[5mm]
 & + \dfrac{\e^2}{3} \, (1 - u)^4 \, g_p(R_p(y, E_i) y, E_j) + \dfrac{\e^2}{3} \sum \limits_{\mu=k+1}^{m+1} \Big( W^\mu \,  g_p(R_p(E_\mu, E_i) y, E_j)  \\[5mm] 
 & + W^\mu \, g_p(R_p(y,E_i)E_\mu,E_j) + \partial_{y_i} W^\mu g_p(R_p(y,E_\mu)y,E_j) + \partial_{y_j} W^\mu \, g_p(R_p(y,E_i)y,E_\mu) \Big) \\[5mm]
 & + \dfrac{\e^3}{6} \, g_p( \nabla_y R_p(y,E_i) y, E_j) + \dfrac{\e^4}{20} \, g_p( \nabla_y \nabla_y R_p(y, E_i) y, E_j ) \\[5mm]
 & + \dfrac{ 2 \e^4}{45} \sum \limits_{l=1}^{k+1} g_p( R_p(y, E_i ) y, E_l) \, g_p( R_p(y, E_i) y, E_l) \\[5mm] 
 & + \dfrac{ 2 \e^4}{45} \sum \limits_{\mu=k+1}^{m+1} g_p( R_p(y, E_i ) y, E_\mu) \, g_p( R_p(y, E_i) y, E_\mu) + \mathcal O(\e^5) 
\end{array} $ }
\end{center}

\medskip

Putting $ y = r \Theta $, $ r \in (0,1) $ we calculate the volume element of $Q_\e(\Pi_p)$:
\begin{center}
 \scalebox{0.9}{ $ \begin{array}{ll}
  \e^{-(k+1)} \, \sqrt{\det g^Q} & = 1 - (k+1) \phi + \dfrac{k(k+1)}{2} \, \phi^2 + \sum \limits_{\mu=k+1}^{m+1} \dfrac{1}{2} | \nabla_{S^k} W^\mu |^2 \\[5mm] 
  & - \dfrac{\e^2}{6} \, r^2 \, \mathcal Ric_{k+1}(\Pi_p)(\Theta,\Theta) + \dfrac{\e^2}{6} \, r^2 \, (k+3) \phi \, \mathcal Ric_{k+1}(\Pi_p)(\Theta,\Theta)  
  \\[5mm] 
  & - \dfrac{\e^2}{3} \, r^2 \sum \limits_{i=1}^{k+1} \sum \limits_{\mu=k+1}^{m+1} \Big( W^\mu \, g_p( R_p(\Theta, E_i, E_\mu,E_i) + \partial_{y_i} W^\mu \, g_p(R_p(\Theta,E_i)\Theta, E_\mu) \Big) \\[5mm]
  & \dfrac{\e^3}{12} \, r^3 \, \nabla_\Theta \, \mathcal Ric_{k+1}(\Pi_p)(\Theta,\Theta)  - \dfrac{\e^4}{40} \, r^4 \, \nabla_\Theta^2 \, \mathcal Ric_{k+1}
  (\Pi_p)(\Theta,\Theta) \\[5mm] 
  & + \dfrac{\e^4}{72} \, r^4 \left( \mathcal Ric_{k+1}(\Pi_p)(\Theta,\Theta) \right)^2  - \dfrac{\e^4}{180} \, r^4 \, \sum \limits_{i,j = 1}^{k+1} g_p(R_p(\Theta,E_i)\Theta, E_j) ^2 \\[5mm] 
  & + \dfrac{\e^4}{45} \, r^4 \, \sum \limits_{i = 1}^{k+1} \, \sum \limits_{\mu = k+1}^{m+1} g_p(R_p(\Theta,E_i)\Theta, E_\mu )^2 + 
\mathcal O_p(\e^5). 
 \end{array} $ }
\end{center}

\medskip

\subsection{Expansion of the energy functional} Collecting the results above gives that 
\begin{center}
 \scalebox{0.9}{ $ \begin{array}{ll}
 \e^{-k} \Big( & \mathrm{Vol} (K_\e(\Pi_p)) - \dfrac{k}{\e} \, \mathrm{Vol}(Q_\e(\Pi_p)) \ \, \Big) \\[5mm] 
 & = \dfrac{1}{k+1} \, \mathrm{Vol}(S^k) - \dfrac{\e^2}{2} \dfrac{1}{k+3}  \displaystyle \int_{S^k} \, \mathcal Ric_{k+1}(\Pi_p)(\Theta, \Theta) \, d \sigma 
 + \displaystyle \int_{S^k} \dfrac{\e^2}{6} \, \mathcal Ric_{k+1} (\Pi_p) \, \phi \, d \sigma \\[5mm] 
 & + \e^4 \, \dfrac{5}{k+5} \displaystyle \int_{S^k} \Big[ - \dfrac{1}{40} \, \nabla_\Theta^2 \, \mathcal Ric_{k+1}(\Pi_p)(\Theta,\Theta) \\[5mm] 
 & + \dfrac{1}{72} \, \left( \mathcal Ric_{k+1}(\Pi_p)(\Theta,\Theta) \right)^2  - \dfrac{1}{180} \, \sum \limits_{i,j = 1}^{k+1} g_p(R_p(\Theta,E_i)\Theta, 
 E_j) ^2 \\[5mm] 
 & + \dfrac{1}{45} \, \sum \limits_{i = 1}^{k+1} \, \sum \limits_{\mu = k+1}^{m+1} g_p(R_p(\Theta,E_i)\Theta, E_\mu )^2 \Big] d \sigma \\[5mm]
 & + \sum \limits_{\mu=k+1}^{m+1} \dfrac{k}{2} \, \displaystyle \int_{B^{k+1}} W^\mu \, \Delta_{B^{k+1}} \, W^\mu \, dy \\[5mm] 
 & + \dfrac{\e^2}{3} \, k \, \sum \limits_{i=1}^{k+1} \displaystyle \int_{B^{k+1}} \Big( W^\mu \, g_p( R_p(\Theta, E_i, E_\mu, E_i) + \partial_{y_i} W^\mu \, R_p(\Theta,E_i,\Theta, E_\mu) \Big) + \mathcal O(\e^5)
 \end{array} $ }
\end{center}

\medskip

We now recall some identities. First, 
\begin{center} \scalebox{0.9}{ $ \displaystyle \int_{S^k}(\Theta^i)^2 \, d \sigma = \dfrac{1}{k+1} \, \mathrm{Vol}(S^k),  $ } \end{center}
\medskip
\begin{center} \scalebox{0.9}{ $ \int_{S^k} (\Theta^i)^4 \, d \sigma = 3 \displaystyle \int_{S^k}(\Theta^i \, \Theta^j)^2 \, d \sigma = \dfrac{3}{(k+1)(k+3)} \, \mathrm{Vol}(S^k) $ } \end{center}
\medskip
and second, if $ a_{ijpq} \in \RR$ $i,j,p,q = 1, \ldots, k+1$,  then
\medskip
\begin{center}
\scalebox{0.9}{ $ \begin{array}{ll} \sum \limits_{p,q,l,n=1}^{k+1} \, \displaystyle \int_{S^k} a_{pqln} \, \Theta^p \, \Theta^q \, \Theta^l \, \Theta^n \, d \sigma & = \frac{3}{(k+1)(k+3)} \, \mathrm{Vol}(S^k) \, \sum \limits_{i=1}^{k+1} \, a_{pppp} \\[3mm] 
 & + \frac{1}{(k+1)(k+3)} \, \mathrm{Vol}(S^k) \, \sum \limits_{q \neq p = 1}^{k+1} \left( a_{ppqq} + a_{pqpq} + a_{pqqp} \right) \\[3mm]\
 & = \frac{1}{(k+1)(k+3)} \, \mathrm{Vol}(S^k) \, \sum \limits_{p,q = 1}^{k+1} \left( a_{ppqq} + a_{pqpq} + a_{pqqp} \right)  \end{array}. $ }
\end{center}

\medskip

We now calculate each term: 

\medskip

\begin{center}
\scalebox{0.9}{ $ \begin{array}{ll}
 \displaystyle \int_{S^k} \mathcal Ric_{k+1}(\Pi_p)(\Theta,\Theta) \, d \sigma & = \sum \limits_{i,j = 1}^{k+1} \displaystyle \int_{S^k} \mathcal Ric_{k+1}(\Pi_p)(E_i,E_j) 
 \, \Theta^k \, \Theta^l \, d \sigma \\[5mm]
 & = \sum \limits_{i=1}^{k+1} \mathcal Ric_{k+1}(\Pi_p)(E_i,E_i) \, ( \Theta^i )^2 \, d \sigma \\[5mm]
 & = \dfrac{1}{k+1} \, \mathrm{Vol}(S^k) \, \mathcal R_{k+1}(\Pi_p) ;
\end{array} $ }
\end{center}

\medskip

\begin{center}
\scalebox{0.9}{ $ \begin{array}{ll}
 \displaystyle \int_{S^k} ( \mathcal Ric_{k+1}(\Pi_p)(\Theta,\Theta) )^2 d \sigma  
 & =  \dfrac{1}{(k+1)(k+3)} \, \mathrm{Vol}(S^k) \, \Big( 2 \sum \limits_{i,j=1}^{k+1} \, (\mathcal Ric_{k+1}(\Pi_p)(E_i,E_j))^2 \Big) \\[5mm] 
 & + \sum \limits_{i,j=1}^{k+1} \mathcal Ric_{k+1}(\Pi_p)(E_i,E_i) \, \mathcal Ric_{k+1}(\Pi_p)(E_j,E_j) \\[5mm]
 & =  \dfrac{1}{(k+1)(k+3)} \, \mathrm{Vol}(S^k) \left( 2 \, \left\| \mathcal Ric_{k+1}(\Pi_p) \right\|^2 + \mathcal R_{k+1}(\Pi_p)^2 \right) ; 
\end{array} $ }
\end{center}

\medskip

\begin{center}
\scalebox{0.9}{ $ \begin{array}{ll}
 \sum \limits_{\alpha,\beta = 1 }^{k} \displaystyle \int_{S^k} & g_p(R_p(\Theta,\Theta_\alpha)\Theta, \Theta_\beta) ^2 \, d\sigma \\[5mm] 
 & = \sum \limits_{i,j = 1}^{k+1} \displaystyle \int_{S^k} g_p(R_p(\Theta,E_i)\Theta,E_j)^2 \, d \sigma \\[5mm]
 & = \dfrac{1}{(k+1)(k+3)} \mathrm{Vol}(S^k) \sum \limits_{i,j,p,q = 1}^{k+1} \left( R_{ipjq}^2 + R_{ipjp} \, R_{iqjq} + R_{ipjq} \, R_{iqjp} \right) 
 \\[5mm]
 & = \dfrac{1}{(k+1)(k+3)} \, \mathrm{Vol}(S^k) \left( \left\| \mathcal Ric_{k+1}(\Pi_p)  \right\|^2 + \dfrac{3}{2} \left\| R_{k+1}(\Pi_p) \right\|^2 \right);
\end{array} $ }
\end{center}
 
\medskip
 
(we use here that $ R_{ijpq}^2 = ( R_{ipjq} - R_{iqjp} )^2 = R_{ipjq}^2 + R_{iqjp}^2 - 2 \, R_{ipjq} \, R_{iqjp} $); 

\medskip

\begin{center}
\scalebox{0.9}{ $ \begin{array}{ll}
 \sum \limits_{\alpha=1}^k \, \sum \limits_{\mu = k+1}^{m+1} & \displaystyle \int_{S^k} g_p(R_p(\Theta,\Theta_\alpha) \Theta, E_\mu) ^2 \, d\sigma \\[5mm] 
 & = \sum \limits_{i = 1}^{k+1} \, \sum \limits_{\mu = k+1}^{m+1} \displaystyle \int_{S^k} g_p (R_p(\Theta, E_i)\Theta,E_\mu)^2 \, d \sigma \\[5mm]
 & = \dfrac{1}{(k+1)(k+3)} \, \mathrm{Vol}(S^k) \left( \left\| \mathcal Ric_{k+1}^\bot(\Pi_p)  \right\|^2 + \dfrac{3}{2} \left\| R_{k+1}^\bot(\Pi_p) \right\|^2 \right);
\end{array} $ }
\end{center}

\medskip

\begin{center}
\scalebox{0.9}{ $ \begin{array}{ll}
 \displaystyle \int_{S^k} \nabla_\Theta^2 \, \mathcal Ric_{k+1}(\Theta,\Theta) \, d \sigma & = \dfrac{1}{(k+1)(k+3)} \sum \limits_{i,j=1}^{k+1} \Big( 
 \nabla^2_{E_i} \, \mathcal Ric_{k+1}(\Pi_p)(E_j,E_j) \\[5mm] 
 & + 2 \, \nabla_{E_i} \nabla_{E_j} \, \mathcal Ric_{k+1}(E_i,E_j) \Big) \\[5mm]  
 & = \dfrac{2}{(k+1)(k+3)} \, \mathrm{Vol}(S^k) \, \Delta_{k+1}^g \mathcal R_{k+1}(\Pi_p);
\end{array} $ }
\end{center}

\medskip

\begin{center}
\scalebox{0.9}{ $ \begin{array}{ll}
- \sum \limits_{\mu=k+1}^{m+1} \displaystyle \int_B^{k+1} | \nabla_{B^{k+1}} W^\mu |^2 \, dy & = \sum \limits_{\mu=k+1}^{m+1} \displaystyle \int_{B^{k+1}} W^\mu \, \Delta_{B^{k+1}} W^\mu 
 \, dy \\[5mm]
 & = - \dfrac{2 \, \e^4}{9} \, \dfrac{1}{k+3} \sum \limits_{\mu=k+1}^{m+1} \displaystyle \int_{B^{k+1}} \sum \limits_{i,j=1}^{k+1} R_{iji\mu}^2 \, (y^j)^2 \, (1 - |y|^2) \, dy \\[5mm]
 & = - \dfrac{2 \, \e^4}{9} \, \dfrac{1}{(k+3)(k+1)} \, \mathrm{Vol}(S^k) \, \| \mathcal Ric_{k+1}^\perp \|^2 \left( \frac{1}{k+3} - \frac{1}{k+5} \right) \\[5mm]
 & = - \dfrac{\e^4}{9} \, \dfrac{4}{(k+1)(k+3)^2(k+5)} \mathrm{Vol}(S^k) \, \| \mathcal Ric_{k+1}^\perp \|^2;
\end{array} $ }
\end{center}

\medskip

\begin{center}
\scalebox{0.9}{ $ \begin{array}{ll}
 \sum \limits_{\mu=1}^{k+1} \displaystyle \int_{B^{k+1}} W^\mu \sum \limits_{i,p=1}^{k+1} R_{ipi\mu} \, y^p \, dy & = - \dfrac{\e^2}{3} \, \dfrac{1}{(k+3)} \,\displaystyle \int_{B^{k+1}} R^2_{iji\mu} \, (y^j)^2( 1 - |y|^2 ) \, dy \\[3mm]
  & = - \dfrac{\e^2}{3} \, \dfrac{2}{(k+1)(k+3)^2(k+5)} \mathrm{Vol}(S^k) \, \| \mathcal Ric_{k+1}^\perp \|^2;
\end{array} $ }
\end{center}

and

\begin{center}
\scalebox{0.9}{ $ \begin{array}{ll}
 \sum \limits_{\mu=k+1}^{m+1} & \displaystyle \int_{B^{k+1}} \partial_{y^i} \, W^\mu R_{piq\mu} \, y^p \, y^q \, dy \\[5mm]
 & = - \dfrac{\e^2}{3} \, \dfrac{1}{(k+3)} \sum \limits_{\mu=k+1}^{m+1} \displaystyle \int_{B^{k+1}} \sum \limits_{i, p, q=1}^{k+1} \Big( \mathcal Ric(\Pi_p)_{i \mu}^\bot \, R_{piq\mu} \, y^p \, y^q \, (1 - |y|^2) \\[5mm] 
 & - 2 \, \sum \limits_{j=1}^{k+1} \mathcal Ric(\Pi_p)_{j\mu}^\bot \, R_{piq\mu} \, y^j \, y^i \, y^p \, y^q  \Big) \, dy \\[5mm]
 & = \dfrac{\e^2}{3} \, \dfrac{2}{(k+1)(k+3)^2(k+5)} \mathrm{Vol}(S^k) \Big[ - \| \mathcal Ric_{k+1}^\perp \|^2 \\[5mm] 
 & + \sum \limits_{p,q = 1}^{k+1} \sum \limits_{\mu=k+1}^{m+1} \left( \mathcal Ric(\Pi_p)_{p \mu}^\bot \, R_{q p q \mu } + \mathcal Ric(\Pi_p)_{q\mu}^\bot \, R_{p p  q \mu} + \mathcal Ric(\Pi_p)_{q\mu}^\bot \, R_{ q p p \mu} \right) \Big] \\[5mm]
 & = -\dfrac{2 \,\e^2}{3} \, \dfrac{1}{(k+1)(k+3)^2(k+5)} \, \mathrm{Vol}(S^k) \, \| \mathcal Ric_{k+1}^\perp \|^2. 
\end{array} $ }
\end{center}

This gives finally

\begin{center}
\scalebox{0.9}{ $ \begin{array}{ll}
  \dfrac{ (k+1) \, \mathcal E(\Pi_p)}{ \e^k \, \mathrm{Vol}(S^k)} & =  1 - \dfrac{\e^2}{2} \dfrac{1}{k+3} \, \mathcal R_{k+1}(\Pi_p) \\[5mm]
 & + \dfrac{\e^4}{72} \, \dfrac{1}{(k+3)(k+5)} \Big( 8 \, \| \mathcal Ric_{k+1}(\Pi_p) \|^2 - 18 \, \, \Delta_{k+1}^g \mathcal R_{k+1}(\Pi_p) - 3 \, \| R_{k+1}(\Pi_p) \|^2 \\[5mm] 
 & + 5 \, \mathcal R_{k+1}(\Pi_p)^2  + 8 \, \| \mathcal Ric_{k+1}^\perp(\Pi_p) \|^2 + 12 \, \| R^\perp_{k+1}(\Pi_p) \|^2 \Big) \\[5mm]
 & + \dfrac{\e^4}{18} \Big( \dfrac{2}{k(k+2)} \mathcal R_{k+1}^2(\Pi_p) - \dfrac{1}{(k+2)(k+3)} \mathcal R_{k+1}^2(\Pi_p) \\[5mm] 
 & - \dfrac{2}{(k+2)(k+3)} \| \mathcal Ric_{k+1}(\Pi_p) \|^2 - \dfrac{12k}{(k+3)^2(k+5)} \| \mathcal Ric_{k+1}^\bot \|^2 \Big) + \mathcal O(\e^5),
\end{array} $ }
\end{center}
or after simplification,
\begin{center}
\scalebox{0.9}{ $ \begin{array}{ll}
 \dfrac{(k+1) \, \mathcal E(\Pi_p)}{ \e^k \, \mathrm{Vol}(S^k)} & = 1 - \dfrac{\e^2}{2} \dfrac{1}{k+3} \, \mathcal R_{k+1}(\Pi_p) \\[5mm]
 & + \dfrac{\e^4}{72} \, \dfrac{1}{(k+3)(k+5)} \Big( 8 \, \| \mathcal Ric_{k+1}(\Pi_p) \|^2 - 18 \, \, \Delta_{k+1}^g \mathcal R_{k+1}(\Pi_p) - 3 \, \| R_{k+1}(\Pi_p) \|^2 \\[5mm] 
 & + 5 \, \mathcal R_{k+1}(\Pi_p)^2  + 8 \, \| \mathcal Ric_{k+1}^\perp(\Pi_p) \|^2 + 12 \, \| R^\perp_{k+1}(\Pi_p) \| \Big) \\[5mm]
 & + \dfrac{\e^4}{18} \, \dfrac{1}{(k+2)(k+3)} \Big( \frac{k+6}{k} \, \mathcal R_{k+1}^2(\Pi_p) - \| \mathcal Ric_{k+1}(\Pi_p) \|^2 \\[5mm] 
 & - \dfrac{12 \, k(k+2)}{(k+3)(k+5)} \, \| \mathcal Ric_{k+1}^\perp(\Pi_p) \|^2 \Big) + \mathcal O(\e^5) \\[5mm]
 & = 1 - \dfrac{\e^2}{2} \dfrac{1}{k+3} \, \mathcal R_{k+1}(\Pi_p) + \dfrac{\e^4}{2(k+3)} \, \mathbf r(\Pi_p) + \mathcal O(\e^5).
\end{array} $ }
\end{center}

\section{Problems}

The results above produce a collection of $k$-dimensional spheres. It is reasonable to suspect that there are other compact 
$k$-dimensional embedded constant mean curvature submanifolds in $\mathbb R^n$ ? Find some other examples! 

\medskip

\noindent Is it possible to build noncompact $k$-dimensional (nonzero) constant mean curvature submanifolds which are 
not contained in a $(k+1)$-dimensional subspace ?  For zero mean curvature submanifolds, the half plane, a half helicoid 
(observe that there are two ways to cut the helicoid in half along a straight line) which has boundary a straight line, 
and a fundamental piece of a Riemann surface, whose boundary are $2$ parallel straight lines, are nontrivial examples.
Are there any analogues in this setting? 

\medskip

\noindent It should follow by unique continuation that if  $K = \partial Q = \partial Q'$ is a constant mean curvature submanifold,
with $H_K \neq 0$, then $Q = Q'$. When $Q$ is a hypersurface,  so $K$ has codimension $2$, this is true by the Hopf boundary
maximum principle.


\begin{thebibliography}{99}

\bibitem{A} F. Almgren, {\em Optimal Isoperimetric inequalities}, Bulletin of the AMS, Vol. 13,  2, (1985),  123-126.

\bibitem{DHL} U. Dierkes and S. Hildebrandt, and H. Lewy, {\em  On the analyticity of minimal surfaces at movable boundaries of prescribed length}. J. Reine Angew. Math. 379 (1987), 100-114. 

\bibitem{H} D. Hoffman, {\em Surfaces in constant curvature manifolds with parallel mean curvature vector field}. Bull. Am. Math. Soc. 78, 247 (1972).

\bibitem{L} H. B. Lawson, Jr., { \em Lectures on minimal submanifolds }, Berkeley CA. {1980}

\bibitem{Mah-Maz-Pac} F. Mahmoudi, R. Mazzeo and F. Pacard. Constant mean curvature hypersurfaces condensing along a submanifold.  Geom. Funct. Anal. 16, no 4, (2006), 924-958

\bibitem{Pac-Xu} F. Pacard et X. Xu. {\em Constant mean curvature spheres in Riemannian manifolds}. Manuscripta Math., 128 (3), 275-295, (2009)

\bibitem{Sch-Yau} R. Schoen, S.T. Yau, { \em Lectures on Differential Geometry }, International Press (1994)

\bibitem{W} T.J. Willmore, {\em Riemannian geometry}. Clarendon Press (1997).

\bibitem{Ye} R. Ye, {\em Foliation by constant mean curvature spheres}, Pacific J. Math. 147 (1991), no. 2, 381Ð396.

\end{thebibliography}
\end{document}